\documentclass[11pt,a4paper]{amsart}
\usepackage{amsmath,amssymb, amsbsy}
\usepackage{subfigure}
\usepackage{graphpap,latexsym,epsf}
\usepackage{color,psfrag}
\usepackage[dvips]{graphicx}

\newcommand{\Hmm}[1]{\leavevmode{\marginpar{\tiny%
$\hbox to 0mm{\hspace*{-0.5mm}$\leftarrow$\hss}%
\vcenter{\vrule \qquad depth 0.1mm height 0.1mm width \the\marginparwidth}%
\hbox to
0mm{\hss$\rightarrow$\hspace*{-0.5mm}}$\\\relax\raggedright #1}}}

\usepackage[english]{babel}
\usepackage{times}
\usepackage{graphicx}
\usepackage{amscd}
\usepackage{amsmath}
\usepackage{amsfonts}
\usepackage{amssymb}
\usepackage{amsthm}
\usepackage{latexsym}
\newtheorem{theorem}{Theorem}[section]
\newtheorem{proposition}[theorem]{Proposition}
\newtheorem{lemma}[theorem]{Lemma}

\newtheorem{definition}[theorem]{Definition}

\newtheorem{remark}[theorem]{Remark}

\def\neweq#1{\begin{equation}\label{#1}}
\def\endeq{\end{equation}}

\def\ds{\displaystyle}

\newcommand{\R}{\mathbb{R}}
\newcommand{\N}{\mathbb{N}}
\newcommand{\HH}{\mathbb{H}}
\newcommand{\eps}{\varepsilon}

\newcommand{\EC}{\overset{E}{\longrightarrow}}
\newcommand{\EEC}{\overset{EE}{\longrightarrow}}
\newcommand{\CC}{\overset{C}{\longrightarrow}}

\begin{document}

\title{   Spectral stability of the  Steklov problem}

\author{Alberto Ferrero, Pier Domenico Lamberti}

\address{\hbox{\parbox{5.7in}{\medskip\noindent{Alberto Ferrero, \\
Universit\`a del Piemonte Orientale, \\
        Dipartimento di Scienze e Innovazione Tecnologica, \\
        Viale Teresa Michel 11, 15121 Alessandria, Italy. \\[5pt]
Pier Domenico Lamberti, \\
Universit\`a di Padova, \\
Dipartimento di Matematica `Tullio Levi-Civita', \\
Via Trieste 63, 35121 Padova, Italy. \\[3pt]
        \em{E-mail addresses: }{\tt alberto.ferrero@uniupo.it, lamberti@math.unipd.it}}}}}

\date{\today}

\maketitle

\begin{abstract} This paper investigates the stability properties of the spectrum of the
classical Steklov problem under domain perturbation. We find
conditions which guarantee the spectral stability and
we show their optimality. We emphasize the fact that our spectral
stability results also involve convergence of eigenfunctions in a
suitable sense according with the definition of connecting system
by \cite{Vainikko}. The convergence of eigenfunctions can be
expressed in terms of the $H^1$ strong convergence. The arguments
used in our proofs are based on an appropriate definition of
compact convergence of the resolvent operators associated with the
Steklov problems on varying domains.

In order to show the optimality of our conditions we present
alternative assumptions which give rise to a degeneration of the
spectrum or to a discontinuity of the spectrum in the sense
that the eigenvalues converge to the eigenvalues of a limit problem
which does not coincide with the Steklov problem on the limiting
domain.
\end{abstract}

\vspace{11pt}

\noindent {\bf Keywords:} Steklov boundary conditions,  spectral
stability, domain perturbations, boundary homogenization.

\vspace{6pt} \noindent {\bf 2000 Mathematics Subject
Classification: 35J25, 35P15, 35M12, 49R05}

\section{Introduction} \label{s:introduction}
In this paper we consider the spectral stability of the classical Steklov problem \cite{Steklov}, namely
\begin{equation} \label{eq:Steklov}
\begin{cases}
\Delta u=0, & \qquad \text{in } \Omega \, , \\
u_{\nu }=\lambda u, & \qquad \text{on } \partial\Omega \, , \\
\end{cases}
\end{equation}
where $u_\nu$ denotes the normal derivative of $u$ on
$\partial\Omega$ and $\lambda$ is a real parameter, see
\cite{girpol} for a survey on this subject. 
Here and in the sequel, the domain for the Steklov problem will always be a bounded domain in $\R^N$ with $N\ge 2$ with appropriate conditions on the regularity of its boundary.

By a solution of \eqref{eq:Steklov} we mean a function $u\in H^1(\Omega)$ such that
\begin{equation} \label{eq:Steklov-var}
\int_\Omega \nabla u \nabla v \, dx=\lambda \int_{\partial\Omega} u v \, dS \, ,
\qquad \text{for any } v\in H^1(\Omega) \, ,
\end{equation}
where  $H^1(\Omega)$ denotes the standard Sobolev space of functions in $L^2(\Omega)$ with first order weak derivatives in $L^2(\Omega)$.
It is well-known that \eqref{eq:Steklov} can be interpreted as an
eigenvalue problem with respect to the parameter $\lambda$: we
recall that $\lambda$ is an eigenvalue for \eqref{eq:Steklov} if
\eqref{eq:Steklov} admits a nontrivial solution in the sense of
\eqref{eq:Steklov-var}. That  nontrivial solution is called
eigenfunction for $\lambda$.

We recall that the set of the eigenvalues of \eqref{eq:Steklov} is
a countable set of isolated nonnegative real  numbers which may be
ordered in an increasing sequence diverging to $+\infty$.
As customary,  we agree
to  repeat each eigenvalue as many
times as its multiplicity:
\begin{equation} \label{eq:sequence}
0=\lambda_0<\lambda_1\le \lambda_2\le \dots \le \lambda_n \le
\dots
\end{equation}
We also recall that the first positive eigenvalue $\lambda_1$
admits the following variational characterization
\begin{equation} \label{eq:lambda1}
\lambda_1=\inf_{v\in \HH(\Omega)} \frac{\int_\Omega |\nabla v|^2 \, dx}{\int_{\partial\Omega}
v^2 \, dS}
\end{equation}
where $\HH(\Omega):=\left\{v\in H^1(\Omega) \setminus
H^1_0(\Omega):\int_\Omega v \, dx=0\right\}$. By a classical
argument it can be shown that all the other successive eigenvalues
admit a minimax characterization as shown in more details in
Section \ref{ss:minimax}.

In this article we study the stability of the eigenvalues
and eigenfunctions of \eqref{eq:Steklov} with respect to domain
perturbation. More precisely if $\{\Omega_\eps\}_{0<\eps\le
\eps_0}$ is a family of domains which converges to a domain
$\Omega$ in a suitable sense, denoting  by
$\{\lambda_n^\eps\}_{n=1}^\infty$, $\{\lambda_n\}_{n=1}^\infty$
the eigenvalues of the Steklov problem in $\Omega_\eps$ and
$\Omega$ respectively, we say that we have stability for the
eigenvalues, if for any $n\ge 1$, $\lambda_n^\eps\to \lambda_n$ as
$\eps\to 0$. We are also interested in what we will call spectral
convergence of $S_{\eps}$ to $S$ where $S_{\eps}$, $S$ denote
the resolvent operators associated with \eqref{eq:Steklov}
 in $\Omega_\eps$ and $\Omega$ respectively, see Section
\ref{s:functional-setting-B} for the precise definitions of these
operators. Exploiting the results contained in \cite{Vainikko},
one can deduce that spectral convergence of those operators
implies not only convergence of eigenvalues but also convergence
of the corresponding eigenfunctions in a suitable sense, see
Theorem \ref{vaithm} for more details.

We note that the behaviour of the eigenvalues of \eqref{eq:Steklov} subject to  domain perturbation
has been discussed in \cite{Bo} which provides sharp  conditions for their stability.
Sufficient conditions ensuring the stability of the resolvent operators (associated with the corresponding  Dirichlet-to-Neumann map) in a class of
star-shaped domains are given in \cite{terelst} where the question whether one could obtain the same results in a general setting is considered  ``out of reach''.
We  also cite the very recent paper \cite{buhema} concerning the asymptotic behaviour of the Steklov problem on dumbbell domains.

The aim of the present paper is to study not only the stability of the eigenvalues as done in \cite{Bo} but also the stability of  the eigenfunctions,  and to give an answer to the question
raised in \cite{terelst}     by proving stability results for the resolvent operators mentioned above in a more general class of domains.  Moreover,  we prove that our conditions are sharp by analysing the behaviour of the resolvent operators  in a limiting case and we study the degeneration phenomena  appearing when the strength
 of the boundary perturbations exceeds the threshold corresponding to that case.  We also apply our results to the homogenization of problem \eqref{eq:Steklov}  subject to periodic boundary perturbations.

 Our results are in the spirit of the  results of \cite{arbru10, ArBru,  arbru06, DancerDaners} concerning  the solutions of second order linear and nonlinear boundary problems with  Robin type boundary conditions.  Namely, in  \cite{DancerDaners} the authors consider the problem
 \begin{equation}\label{robin}\nonumber
 \left\{
 \begin{array}{ll}
 -\Delta u=f(u),& {\rm in}\ \Omega,\\
 u_{\nu}+\beta u=0,& {\rm in}\ \partial\Omega,
 \end{array}
 \right.
 \end{equation}
 where $\beta$ is a fixed positive constant and provide  stability and instability results for the solutions $u$ upon  perturbation of $\Omega$. In particular, they  identify the  conditions on the  perturbations of $\Omega$ which cause  the degeneration of Robin boundary conditions to Dirichlet boundary conditions. They also discuss the  case where a  deformation of the coefficient $\beta$ appears in the limiting problem. Similar results have been obtained
 in \cite{arbru10, ArBru,  arbru06} for nonlinear boundary conditions of the form $u_{\nu}+g(x,u)=0$.

We observe that the resolvent  operators $S_\eps$ act on functional spaces
which depend on the parameter $\eps$ so that, in order to provide
a reasonable notion of compact convergence for operators acting on
different functional spaces, we make use of suitable ``connecting
systems'' which allow to pass from the varying Hilbert spaces
defined on $\Omega_{\eps}$ to the limiting fixed Hilbert space
defined on $\Omega$.
This approach is made possible thanks to a number of
notions and results which go back to the works of F.
Stummel~\cite{Stu} and G. Vainniko~\cite{Vainikko} and which have
been further implemented in \cite{ArCaLo, CarPisk}. See also the
recent paper \cite{Bogli}. All these notions are discussed in
detail in Section~ \ref{ss:convergenceint}. These abstract results
are based on the notion of $E$-compact convergence of operators,
possibly acting on different Hilbert spaces.

In order to prove stability results and  the corresponding $E$-compact convergence of the operators $S_{\eps}$,
 we consider domains $\Omega_{\eps}$,
$\Omega $ belonging to uniform classes of domains with $C^{0,1}$
boundaries - see Definition~\ref{d:atlas-B} - and we  require that
the boundaries of $\Omega_{\eps}$ converge to the boundary of
$\Omega$ in the sense of \eqref{eq:assumptions}-\eqref{eq:conv-per}. We observe that if the
boundaries of $\Omega_{\eps}$ converge to the boundary of $\Omega$
in $C^1$ then conditions \eqref{eq:assumptions}-\eqref{eq:conv-per} are satisfied.
Conditions \eqref{eq:assumptions}-\eqref{eq:conv-per} make
possible the construction in Section~\ref{subsecoperatorsE} of a
family of linear continuous operators $E_{\eps}:H^1(\Omega)\to
H^1(\Omega_{\eps})$, i.e. the connecting system, which allows us to
treat operators defined on $H^1(\Omega_\eps)$ and $H^1(\Omega)$ simultaneously. In fact, the family of operators
$\{E_\eps\}$ is a connecting system in the sense of \cite{Vainikko} provided that
$H^1(\Omega_\eps)$ is endowed with the equivalent norm
$$
\|u\|_\eps:=\left(\int_{\Omega_\eps} |\nabla u|^2 \, dx+\int_{\partial\Omega_\eps} u^2 \, dS\right)^{1/2} \qquad \text{for any }
u \in H^1(\Omega_\eps) \, .
$$
In particular, the family of operators $\{E_\eps\}$  makes possible the
definition of the notion of $E$-convergence for a sequence
$u_{\eps }\in H^1(\Omega_{\eps})$ to a function $u\in
H^1(\Omega)$, i.e.
\begin{equation}
\label{econvergenza}
\|  u_{\eps }-E_{\eps}u\|_{\eps}\to 0, \ \ {\rm
as}\ \eps \to 0,
\end{equation}
which allows to overcome the problem of having different functional spaces to deal with. We note that, similarly to \cite{FeLa}, the operators $E_{\eps}$  are
constructed by pasting together suitable pull-back operators
defined by means of  appropriate local diffeomorphisms.

In is important to observe that the operators $E_{\eps}$ satisfy the following property: for any
$\eps>0$ there exists an open set $K_{\eps }\subset \Omega\cap
\Omega _{\eps }$ such that $(E_{\eps}u)(x)=u(x)$ for all $x\in
K_{\eps}$ and such that $|(\Omega_{\eps}\cup \Omega ) \setminus
K_\eps |\to 0$ as $\eps \to 0$. (Here and in the sequel $|A|$
denotes the Lebesgue measure of any measurable set $A\subset
\R^N$.) This, combined with  \eqref{econvergenza}, implies the more familiar strong convergence
\begin{equation}
\label{intersection} \|  u_{\eps }-u\|_{H^1(\Omega _{\eps } \cap
\Omega )  }\to 0, \ \ {\rm as}\ \eps \to 0
\end{equation}
as we state in Proposition \ref{intersec}.

The main results of the paper are contained in Section
\ref{s:main-results}; we proceed by describing in details the
meaning of their contents.

The first main result is Theorem \ref{t:main-1} in which we prove
the spectral stability of \eqref{eq:Steklov} under the validity of
conditions \eqref{eq:assumptions}-\eqref{eq:conv-per}. More
precisely we prove the $E$-compact convergence of  $S_{\eps}$ to
$S$ as $\eps\to 0$. On the base of the general
Theorem~\ref{vaithm}, this implies the spectral convergence of
$S_{\eps}$ to $S$, hence the convergence of the eigenvalues and
the $E$-convergence of the eigenfunctions in the sense of
Theorem~\ref{vaithm}. In particular the eigenfunctions converge in
the sense of \eqref{intersection}.

The subsequent results aim to show the optimality of conditions
assumed in Theorem \ref{t:main-1}.

Theorem \ref{t:main-3} can be considered an extension of Theorem
\ref{t:main-1} since we assume again the validity of
\eqref{eq:assumptions} but we replace \eqref{eq:conv-per} with the
weaker condition \eqref{eq:weak-L1}. Indeed, if we assume the
validity of \eqref{eq:assumptions} and \eqref{eq:conv-per}
simultaneously then \eqref{eq:weak-L1} holds true with
$\gamma_j=\sqrt{1+|\nabla_{x'}g_j|^2}$ so that the function
$\gamma$ defined in \eqref{eq:gamma} satisfies $\gamma\equiv 1$ on
$\partial\Omega$. In such a case the eigenvalues $\lambda_n^\eps$
converge to $\lambda_n$ as $\eps\to 0$ and spectral stability is
proved. But whenever the function $\gamma\not\equiv 1$ on
$\partial \Omega$, convergence of the eigenvalues to the natural
limit problem fails to be true thus giving rise to a discontinuity phenomenon. In Proposition \ref{p:comparison}
(ii), we exhibit an explicit example where the function
$\gamma\not \equiv 1$.
We note that,  assumption  \eqref{eq:weak-L1} is, mutatis mutandis,  the condition  used  in \cite[Theorem~4.4]{DancerDaners} and in \cite{ArBru, arbru06}.

Then in Theorem \ref{t:main-2}, we assume the alternative
conditions \eqref{eq:th-main-2-0}-\eqref{eq:th-main-2} and we
prove degeneration of eigenvalues by showing that
$\lambda_n^\eps\to 0$ as $\eps\to 0$ for any $n\ge 1$. An explicit
example in which \eqref{eq:th-main-2-0}-\eqref{eq:th-main-2} hold
true can be found in Proposition \ref{p:comparison} (iii).
We also note that condition \eqref{eq:th-main-2}  was used in  \cite{arbru10, DancerDaners} to prove the degeneration of the Robin problem to the Dirichlet problem.

Finally, we provide a more clear picture of the
conditions contained in Theorems \ref{t:main-1}-\ref{t:main-2} by
considering a particular case in which the domains $\Omega_\eps$
and $\Omega$ are in the form
\begin{equation*}
\Omega_\eps=\{(x',x_N)\in \R^N:x'\in W\, , \ -1<x_N<g_\eps(x')\}
\, , \qquad \Omega=W\times (-1,0)
\end{equation*}
where $W$ is a cuboid or a bounded domain in $\R^{N-1}$ of class
$C^{0,1}$, $g_\eps(x')=\eps^\alpha b(x'/\eps)$ for any $x'\in W$
and $b:\R^{N-1}\to [0,+\infty)$ is a nonconstant $Y$-periodic
function with $Y=\left(-\frac 12,\frac 12\right)^{N-1}$ the unit
cell in $\R^{N-1}$. Note that this type of periodic perturbations is classical in homogenization theory, in particular in the study of boundary homogenization problems, see e.g.,  \cite{ArLa2, ferlam} and the references therein.
In this particular situation, the conditions introduced in Theorems
\ref{t:main-1}-\ref{t:main-2} find a clear representation
depending on the value assumed by the exponent $\alpha$. More
precisely, it is proved in Proposition \ref{p:comparison} that the
assumptions of the three main theorems correspond to the cases
$\alpha>1$, $\alpha=1$ and $0<\alpha<1$ respectively.
Taking into account what  was shown in Proposition
\ref{p:comparison}, the statements of Theorems
\ref{t:main-1}-\ref{t:main-2} can be unified in a single result
contained in Theorem \ref{t:tricotomia}. This theorem shows
spectral stability for $\alpha>1$, degeneration in the case
$0<\alpha<1$ and a  discontinuity phenomenon in the limiting
case $\alpha=1$ in the sense that the eigenvalues of
\eqref{eq:Steklov-modified} converge to the eigenvalues of the
modified problem \eqref{eq:Steklov-weighted-bis} whose eigenvalues
are given by the eigenvalues of \eqref{eq:Steklov-modified}
divided by the constant $C_b=\int_Y \sqrt{1+|\nabla_{x'}b(x')|^2}
\, dx'$.

We note that in our previous paper \cite{FeLa}, we considered the spectral stability of certain Steklov problems for the biharmonic operator.  Although the proof of   Theorem~\ref{t:main-1}  is based on  a  method  similar to the one used for the proof of the corresponding stability results in  \cite{FeLa},
the part of the present paper concerning the discontinuity and the degeneration phenomenon - namely, Theorems~\ref{t:main-3},  \ref{t:main-2} -  is completely different and is   specifically  designed for problem  \eqref{eq:Steklov}.

We conclude this section by explaining how this paper is organized.
In Section \ref{preliminaries} we state some well-known basic
results about spectral stability for operators defined on abstract
Hilbert spaces, we introduce the main assumptions about the
perturbed domains $\Omega_\eps$ and the limit domain $\Omega$, we
contruct the resolvent operator $S$ associated with
\eqref{eq:Steklov}, we construct a connecting system acting from
$H^1(\Omega)$ into $H^1(\Omega_\eps)$ and finally we show that the
classical minimax characterization of eigenvalues applies to the
Steklov spectrum.
Section \ref{s:main-results} is devoted to the statements of the
main results already described above.
Section \ref{s:p-t:main-1} is devoted to the proof of Theorem
\ref{t:main-1} and to the statement and the proof of Proposition
\ref{intersec}.
Sections \ref{s:p-t:main-3}-\ref{s:p-t:main-2} are devoted to the
proofs of Theorems \ref{t:main-3}-\ref{t:main-2} respectively,
Section \ref{s:p-p:comparison} to the proof of Proposition
\ref{p:comparison} and finally Section \ref{s:t-t:tricotomia} to
the proof of Theorem \ref{t:tricotomia}.

\section{Preliminaries and notation} \label{preliminaries}

\subsection{A general approach to spectral stabilty} \label{ss:convergenceint}
The study of the spectral stability of \eqref{eq:Steklov} is
reduced to the study of suitable families  $\{B_\eps\}_{0<\eps\le
\eps_0}$ of non-negative compact self-adjoint operators defined in
Hilbert spaces ${\mathcal{H}}_{\eps}$ associated with the domains
$\Omega_{\eps }$.

In order to follow this approach we recall here the notion of
$E$-convergence. As already mentioned in the Introduction, we
follow the approach of \cite{Vainikko} and the successive
development by \cite{ArCaLo}, \cite{CarPisk}.

According to the notation used in \cite{ArCaLo} (see also
\cite{FeLa}), we denote by ${\mathcal{H}}_\eps$ a family of
Hilbert spaces for $\eps\in [0,\eps_0]$ and we assume that there
exists a family of linear operators $E_\eps:{\mathcal{H}}_0\to
{\mathcal{H}}_\eps$ such that
\begin{equation} \label{eq:norm-convergence}
\|E_\eps u\|_{{\mathcal{H}}_\eps} \overset{\eps\to 0}{\longrightarrow}
\|u\|_{{\mathcal{H}}_0} \, , \qquad \text{for all } u\in {\mathcal{H}}_0 \, .
\end{equation}

\begin{definition} \label{d:E-convergence-ArCaLo} We say that a family
$\{u_\eps\}_{0<\eps\le \eps_0}$, with $u_\eps\in {\mathcal{H}}_\eps$,
$E$-converges to $u\in {\mathcal{H}}_0$ if $\|u_\eps-E_\eps u\|_{{\mathcal{H}}_\eps}\to 0$
as $\eps\to 0$. We write this as $u_\eps \EC u$.
\end{definition}

\begin{definition} \label{d:EE-convergence-ArCaLo}
Let $\{B_\eps\in\mathcal L({\mathcal{H}}_\eps):\eps \in (0,\eps_0]\}$ be a
family of linear and continuous operators. We say that
$\{B_\eps\}_{0<\eps\le \eps_0}$ converges to $B_0\in \mathcal L({\mathcal{H}}_0)$ as $\eps\to 0$
if $B_\eps u_\eps \EC B_0 u$ whenever $u_\eps \EC u$. We write
this as $B_\eps \EEC B_0$.
\end{definition}

\begin{definition} \label{d:precompact-ArCaLo}
Let $\{u_\eps\}_{0<\eps\le \eps_0}$ be a family such that
$u_\eps\in {\mathcal{H}}_\eps$. We say that $\{u_\eps\}_{0<\eps\le
\eps_0}$ is precompact if for any sequence $\eps_n\to 0$ there
exist a subsequence $\{\eps_{n_k}\}$ and $u\in {\mathcal{H}}_0$
such that $u_{\eps_{n_k}} \EC u$.
\end{definition}

\begin{definition} \label{d:CC-convergence-ArCaLo}
We say that $\{B_\eps\}_{0<\eps\le \eps_0}$ with $B_\eps\in
\mathcal L({\mathcal{H}}_\eps)$ and $B_\eps$ compact, converges
compactly to a compact operator $B_0\in \mathcal
L({\mathcal{H}}_0)$ if $B_\eps \EEC B_0$ and for any family
$\{u_\eps\}_{0<\eps\le \eps_0}$ such that $u_\eps\in
{\mathcal{H}}_\eps$, $\|u_\eps\|_{\mathcal H_\eps}=1$, we have
that $\{B_\eps u_\eps\}_{0<\eps\le \eps_0}$ is precompact in the
sense of Definition \ref{d:precompact-ArCaLo}. We write this as
$B_\eps \CC B_0$.
\end{definition}

We now recall some notations already used in \cite{FeLa}. If $B$
is a non-negative compact self-adjoint operator in a infinite
dimensional Hilbert space ${\mathcal{H}}$, its spectrum is a
finite or a countably infinite set and all non-zero elements of
the spectrum are positive eigenvalues of finite multiplicity. When
the spectrum is a countably infinite set, the eigenvalues can be
represented by a non-increasing  sequence $\mu_n$, $n\in \N$, such
that $\lim_{n\to \infty }\mu_n=0$. As usual we agree to  repeat  each eigenvalue in the sequence $\mu_n$, $n\in {\mathbb{N}}$ as many times as its multiplicity.

We also define the notion of generalized eigenfunction: given a
finite set of $m$ eigenvalues $\mu_n, \dots , \mu_{n+m-1}$ with
$\mu_n\ne \mu_{n-1}$ and $\mu_{n+m-1}\ne \mu_{n+m}$, we call
generalized eigenfunction (associated with $\mu_n,  \dots ,
\mu_{n+m-1}$) any linear combination of eigenfunctions associated
with the eigenvalues $\mu_n,  \dots , \mu_{n+m-1}$.

We now state the following theorem which is a simplified rephrased
version of \cite[Theorem~6.3]{Vainikko}, see also
\cite[Theorem~4.10]{ArCaLo}, \cite[Theorem~5.1]{ArrZua},
\cite[Theorem~3.3]{CarPisk} and \cite[Theorem 1]{FeLa}.

 \begin{theorem}
 \label{vaithm}
 Let  $\{B_\eps\}_{0<\eps\le \eps_0}$ and $B_0$ be non-negative compact self-adjoint operators in the Hilbert
 spaces ${\mathcal{H}}_{\eps }$, ${\mathcal{H}}_0$ respectively. Assume that their  eigenvalues are given
 by the sequences $\mu_{n}(\eps )$ and $\mu_{n}(0)$, $n\in \N$, respectively.  If  $B_\eps \CC B_0$ then we have
 spectral convergence of $B_{\eps}$ to $B_0$ as $\eps \to 0$ in the sense that the following statements hold:
 \begin{itemize}
 \item[(i)] For every $n\in \N$ we have $\mu_n(\eps )\to \mu_n(0)$ as $\eps \to 0$.
 \item[(ii)] If $u_n(\eps)$, $n\in \N$, is an orthonormal sequence of eigenfunctions associated with the eigenvalues  $\mu_n(\eps )$ then
 there exists a sequence $\eps_k$, $k\in \N$,  converging to zero and orthonormal sequence of eigenfunctions  $u_n(0)$, $n\in \N$ associated with  $\mu_n(0 )$, $n\in \N$ such that $u_n(\eps_k)\EC u_n(0)$.
 \item[(iii)]  Given  $m$ eigenvalues $\mu_n(0),  \dots , \mu_{n+m-1}(0)$ with
$\mu_n(0)\ne \mu_{n-1}(0)$ and $\mu_{n+m-1}(0)\ne \mu_{n+m}(0)$
 and corresponding orthonormal eigenfunctions $u_n(0),  \dots , u_{n+m-1}(0)$,
 there exist $m$ orthonormal   generalized eigenfunctions   $v_n(\eps ),  \dots , v_{n+m-1}(\eps )$  associated with  $\mu_n(\eps ),  \dots ,
  \mu_{n+m-1}(\eps )$   such that $v_{n+i}(\eps )\EC u_{n+i}(0)$  for all $i=0, 1,\dots , m-1$.
 \end{itemize}
\end{theorem}

\subsection{Classes of domains} \label{opensets}
According to \cite{FeLa}, in order study the  spectral convergence
for the Steklov eigenvalue problem, we shall consider uniform
families of domains with some prescribed common properties. In
this perspective we recall the notion of atlas from
\cite{BuLa}, see also \cite[Section 5]{ArLa2} and \cite[Section
2.2]{FeLa}. According to \cite{ArLa2,BuLa}, given a set $V\subset
\R^N$ and a number $\delta>0$ we write
\begin{equation} \label{eq:def-Vdelta-B}
V_\delta:=\{x\in V:d(x,\partial V)>\delta\} \, .
\end{equation}

\begin{definition} \label{d:atlas-B}
\ \cite[Definition 2.4]{BuLa} Let $\rho>0$, $s,s'\in \N$ with
$s'<s$. Let $\{V_j\}_{j=1}^s$ be a family of open cuboids (i.e.
rotations of rectangle parallelepipeds in $\R^N$) and
$\{r_j\}_{j=1}^s$ be a family of rotations in $\R^N$. We say that
$\mathcal A=(\rho,s,s',\{V_j\}_{j=1}^s,\{r_j\}_{j=1}^s)$ is an
atlas in $\R^N$ with parameters
$\rho,s,s',\{V_j\}_{j=1}^s,\{r_j\}_{j=1}^s$, briefly an atlas in
$\R^N$. Moreover, we say that a bounded domain $\Omega$ in $\R^N$
belongs to the class $C^{k,\gamma}(\mathcal A)$ with $k\in \N$ and
$\gamma\in [0,1]$ if the following conditions are satisfied:
\begin{itemize}
\item[(i)] $\Omega\subset \cup_{j=1}^s (V_j)_\rho$ \ and \
$(V_j)_\rho \cap \Omega\neq \emptyset$ where $(V_j)_\rho$ is meant
in the sense given in \eqref{eq:def-Vdelta-B} ;

\item[(ii)] $V_j\cap \partial\Omega \neq \emptyset$ \ for \
$j=1,\dots,s'$ \ and \ $V_j\cap \partial\Omega=\emptyset$ \ for \
$s'+1\le j\le s$;

\item[(iii)] for $j=1,\dots,s$ \ we have
\begin{align*}
& r_j(V_j)=\{x\in \R^N:a_{ij}<x_i<b_{ij}\, , i=1,\dots,N\}, \quad
j=1,\dots,s ; \\
& r_j(V_j\cap \Omega)=\{x=(x',x_N)\in\R^N:x'\in W_j,
a_{Nj}<x_N<g_j(x')\}, \quad j=1,\dots, s'
\end{align*}
where $x'=(x_1,\dots,x_{N-1})$, $W_j=\{x'\in
\R^{N-1}:a_{ij}<x_i<b_{ij}, i=1,\dots,N-1\}$ and the functions
$g_j\in C^{k,\gamma}(\overline{W_j})$ for any $j\in 1,\dots,s'$ with $k\in
\N\cup \{0\}$ and $0\le \gamma\le 1$. Moreover, for $j=1,\dots,s'$
we assume that $a_{Nj}+\rho\le g_j(x')\le b_{Nj}-\rho$, for all
$x'\in \overline{W_j}$.
\end{itemize}
\noindent Finally we say that $\Omega$ if of class $C^{k,\gamma}$
if it is of class $C^{k,\gamma}(\mathcal A)$ for some atlas
$\mathcal A$. In the sequel $C^{k,0}$ will be simply denoted by $C^k$.
\end{definition}

Let $\Omega$ be a bounded domain in ${\mathbb{R}}^N$ of class $C^{0,1}$. The Hilbert space $H^1(\Omega)$
is naturally endowed with the scalar product
$$
\int_{\Omega} \nabla u \nabla v \, dx+\int_\Omega uv \, dx \qquad \text{for any } u,v\in H^1(\Omega) \, .
$$
However, taking into account the structure of problem
\eqref{eq:Steklov}, it appears reasonable to replace the classical
scalar product of $H^1(\Omega)$ with another one in which the
scalar product in $L^2(\Omega)$ of the two functions $u,v\in
H^1(\Omega)$ is replaced by the scalar product in
$L^2(\partial\Omega)$ of their traces on $\partial\Omega$. In the
next lemma we recall that these two scalar products are equivalent
in $H^1(\Omega)$.

\begin{lemma} \label{complete}
Let $\Omega$ be a bounded domain in ${\mathbb{R}}^N$ of class $C^{0,1}$. Then we have:
\begin{itemize}
\item[(i)] there exists a constant $C(N,\Omega)$ depending only on $N$ and $\Omega$ such that
  \begin{equation*}
    \int_{\partial\Omega} u^2 \, dS \le C(N,\Omega) \left(\int_\Omega |\nabla u|^2 \, dx+\int_{\Omega} u^2 \, dx\right)
    \qquad
    \text{for any } u\in H^1(\Omega) \, ;
    \end{equation*}
more precisely, if $\mathcal A$ is an atlas as in Definition
\ref{d:atlas-B} such that $\Omega$ is of class $C^{0,1}(\mathcal
A)$, the dependence of $C(N,\Omega)$ on $\Omega$ occurs through
the atlas $\mathcal A$ and the $C^{0,1}$ norms of the functions
$g_j$ introduced in the same definition.

\item[(ii)] there exists a constant $C(N,{\rm diam}(\Omega))$ depending only on $N$ and ${\rm diam}(\Omega)$, where ${\rm diam}(\Omega)$ denotes the diameter of $\Omega$, such that
    \begin{equation*}
    \int_\Omega u^2 \, dx\le C(N,{\rm diam}(\Omega)) \left(\int_\Omega |\nabla u|^2 \, dx+\int_{\partial \Omega} u^2 \, dS
    \right) \qquad
    \text{for any } u\in H^1(\Omega) \, ;
    \end{equation*}

\item[(iii)] the following scalar product
\begin{equation*}
\int_\Omega \nabla u \nabla v \, dx+\int_{\partial\Omega} uv \, dS
\qquad \text{for any } u,v\in H^1(\Omega)
\end{equation*}
is equivalent to the original scalar product of $H^1(\Omega)$.

\end{itemize}

\end{lemma}

\begin{proof}
Part (i) of the lemma is a well known result from classical trace theorems, see for example the book by \cite{Necas}. Part (ii)
can be obtained in a classical way by using the divergence formula
and the H\"older-Young inequality. Finally, part (iii) is an
immediate consequence of (i) and (ii).
\end{proof}

Thanks to Lemma \ref{complete} the space $H^1(\Omega)$ may be equivalently endowed with the scalar product
\begin{align} \label{eq:prod-0}
 (u,v)_0:=\int_{\Omega} \nabla u\nabla v \, dx+\int_{\partial\Omega} uv \, dS  \qquad \text{for any } u,v\in
H^1(\Omega)
\end{align}
and the corresponding norm
\begin{align} \label{eq:norm-0}
 \|u\|_0:=(u,u)_0^{1/2}  \qquad \text{for any } u\in
H^1(\Omega) \, .
\end{align}

\subsection{The functional setting} \label{s:functional-setting-B}
Assume that  $\Omega\subset \R^N$ ($N\ge 2$) is a bounded domain
of class $C^{0,1}$.

Similarly to \cite{FeLa}, we introduce the following resolvent
operator $S:H^1(\Omega)\to H^1(\Omega)$ associated with
\eqref{eq:Steklov} which turns out to be a nonnegative
self-adjoint compact operator.

In order to costruct the operator $S$, we first introduce the
operator \ $T:H^1(\Omega)\to (H^1(\Omega))'$ defined by
\begin{equation} \label{eq:def-T}
_{(H^1(\Omega))'} \langle Tu,v\rangle_{H^1(\Omega)}:=\int_\Omega
\nabla u \nabla v \, dx+\int_{\partial\Omega} uv \, dS \qquad
\text{for any } u,v\in H^1(\Omega) \, .
\end{equation}
The operator $T$ is clearly well-defined and continuous. Moreover
by Lemma \ref{complete} (iii) and Lax-Milgram Theorem, we also
deduce that $T$ is invertible and $T^{-1}$ is continuous.

Then we introduce the operator $J:H^1(\Omega)\to (H^1(\Omega))'$ defined by
\begin{equation} \label{eq:def-J}
_{(H^1(\Omega))'} \langle Ju,v \rangle_{H^1(\Omega)} :=
\int_{\partial\Omega} u v \, dS \qquad \text{for any } u,v\in
H^1(\Omega) \, .
\end{equation}
Since the trace map
\begin{align} \label{eq:trace}
 H^1(\Omega) & \mapsto L^2(\partial\Omega) \\
 \notag  u & \mapsto u_{|\partial\Omega}
\end{align}
is well-defined and compact being $\partial\Omega$ Lipschitzian  (see \cite[Theorem 6.2, Chap. 2]{Necas} for more details), then the operator $J$ is also
well-defined and compact.

We are ready to define the operator $S:H^1(\Omega) \to
H^1(\Omega)$ as $S:=T^{-1} \circ J$. Clearly $S$ is a linear
compact operator and moreover it is easy to see that it is also
self-adjoint. Moreover one can show that $\mu\neq 0$ is an
eigenvalue of $S$ with corresponding eigenfunction $u$ if and only
if $\lambda:=\frac 1\mu-1$ is an eigenvalue of \eqref{eq:Steklov}
with corresponding eigenfunction $u$.

\subsection{Domain perturbations and construction of a connecting system} \label{subsecoperatorsE}

Let $\mathcal A$ be an atlas. Let $\{\Omega_\eps\}_{0<\eps\le
\eps_0}$ be a family of domains of class $C^{0,1}(\mathcal A)$
which converges to a fixed domain $\Omega$ of class
$C^{0,1}(\mathcal A)$ in a sense which will be specified below.
For any $0<\eps\le \eps_0$ denote by
\begin{equation*}
S_{\eps}:H^1(\Omega_\eps) \to H^1(\Omega_\eps)
\end{equation*}
the resolvent operators associated with \eqref{eq:Steklov} in
$\Omega_\eps$ according with the definition given in Section
\ref{s:functional-setting-B}.

Since our final purpose will be to apply the abstract results of
Section \ref{ss:convergenceint}, we need to define a family of
operators $E_{\eps}$, which satisfy condition
\eqref{eq:norm-convergence}. In the specific case under
consideration, this means that we have to introduce
 linear operators
$E_\eps:H^1(\Omega)\to H^1(\Omega_\eps)$ such that
\begin{equation} \label{eq:conv-norm}
\|E_\eps u\|_\eps \to \|u\|_0 \qquad \text{as } \eps\to 0\, , \quad \text{for any } u\in
H^1(\Omega) \, ,
\end{equation}
where
\begin{align}  \label{eq:equiv-norms}
& (u,v)_\eps:=\int_{\Omega_\eps} \nabla u\nabla v \, dx+\int_{\partial\Omega_\eps} uv \, dS,  \qquad \text{for any } u,v\in
H^1(\Omega_\eps), \\[7pt]
\notag & \|u\|_\eps=(u,u)_\eps^{1/2}  \qquad \text{for any } u\in
H^1(\Omega_\eps) \, .
\end{align}
We recall that by Lemma \ref{complete} the norms $\|\cdot\|_\eps$
and $\|\cdot\|_0$ are equivalent to the original norms of
$H^1(\Omega_\eps)$ and $H^1(\Omega)$ respectively. For this reason
it will be convenient in the sequel, except when it is otherwise
specified, to endow the spaces $H^1(\Omega_\eps)$ and
$H^1(\Omega)$ with the scalar products $(\cdot,\cdot)_\eps$ and
$(\cdot,\cdot)_0$ respectively.

In order to prove \eqref{eq:conv-norm}, we proceed ad in
\cite{ArLa2}, see also \cite{FeLa}. Denote by $g_j, g_{\eps,j}\in
C^{0,1}(\overline{W}_j)$ the functions corresponding respectively
to $\Omega$ and $\Omega_\eps$ according to Definition
\ref{d:atlas-B}.

Suppose that the following assumptions hold true: for any
$j=1,\dots ,s'$ and $k\in \{1,\dots,N-1\}$
\begin{align} \label{eq:assumptions}
\lim_{\eps\to 0} \|g_{\eps,j}-g_j\|_{L^\infty(W_j)}=0 \, , \quad
\left\|\tfrac{\partial g_{\eps,j}}{\partial
x_k}\right\|_{L^\infty(W_j)}\!\!=\!O(1), \qquad \text{as } \eps\to
0
\end{align}
and
\begin{align} \label{eq:conv-per}
{\rm Per}(\Omega_\eps)\to {\rm Per}(\Omega), \qquad \text{as }
\eps\to 0,
\end{align}
where ${\rm Per}(\Omega_\eps)$ and ${\rm
Per}(\Omega)$ denote the perimeters of the domains $\Omega_\eps$
and $\Omega$ respectively, i.e. the $(N-1)$-dimensional measures
of $\partial\Omega_\eps$ and $\partial\Omega$ respectively.

\begin{remark} We recall that condition \eqref{eq:conv-per}
is a condition already used in \cite{Bo} to prove convergence of
the eigenvalues of the Steklov operator. Moreover by the proof of
\cite[Proposition 3.2]{Bo} one can deduce that, assuming
\eqref{eq:assumptions}-\eqref{eq:conv-per}, the following
pointwise convergence
\begin{equation} \label{eq:Jac}
\sqrt{1+|\nabla_{x'}g_{\eps,j}|^2} \to \sqrt{1+|\nabla_{x'}g_{j}|^2} \quad \text{a.e. in } W_j \quad \text{as } \eps \to 0 \, ,
\end{equation}
holds true.
\end{remark}

Following the construction introduced in \cite{ArLa2} and used in
\cite{FeLa}, we are going to define the family of operators
$\{E_\eps\}_{0<\eps\le \eps_0}$ (up to shrink $\eps_0$ if
necessary) by using a partition of unity and pasting together
suitable pull-back operators associated with local diffeomorphisms
defined on each cuboid of the atlas $\mathcal A$. Note that in the
simplified setting of one single cuboid, partition of unity would
not be required and the operator $E_\eps$ would be simply defined
as in Remark \ref{r:cuboid} below.

Let $\hat k$ be a fixed constant satisfying $\hat k>4$ whose
meaning will be explained below. Let us define
\begin{equation*} 
\kappa_\eps:=\max_{1\le j\le s'}
\|g_{\eps,j}-g_j\|_{L^\infty(W_j)}\, , \qquad k_\eps:= \hat k
\kappa_\eps \, , \qquad \tilde g_{\eps,j}:=g_{\eps,j}-k_\eps \, ,
\end{equation*}
and
\begin{equation*}
K_{\eps,j}:=\{(x',x_N)\in W_j\times
(a_{Nj},b_{Nj}):a_{Nj}<x_N<\tilde g_{\eps,j}(x')\} \qquad \text{
for any } j=1,\dots,s' \, .
\end{equation*}
For any $1\le j\le s'$ we define the map $h_{\eps,j}:r_j(\overline{\Omega_\eps\cap V_j}) \to
\R$
\begin{equation} \label{eq:def-h}
h_{\eps,j}(x',x_N):=
\begin{cases} 0, & \qquad \text{if } a_{jN} \le x_N\le \tilde g_{\eps,j}(x'), \\
\left(g_{\eps,j}(x')-g_j(x')\right)\left(\frac{x_N-\tilde g_{\eps,j}(x')}{g_{\eps,j}(x')-\tilde g_{\eps,j}(x')}\right)^2,
& \qquad \text{if } \tilde g_{\eps,j}(x')< x_N\le g_{\eps,j}(x') \, .
\end{cases}
\end{equation}
We observe that $h_{\eps,j}\in C^{0,1}(r_j(\overline{\Omega_\eps
\cap V_j}))$ and that the map
$\Phi_{\eps,j}:r_j(\overline{\Omega_\eps \cap V_j}) \to
r_j(\overline{\Omega \cap V_j})$ defined by
$\Phi_{\eps,j}(x',x_N):=(x',x_N-h_{\eps,j}(x',x_N))$ is a
homeomorphism of class $C^{0,1}$. When $s'+1\le j\le s$ we define
$\Phi_{\eps,j}:r_j(\overline{V_j})\to r_j(\overline{V_j})$ as the
identity map.

Consider now a partition of unity $\{\psi_j\}_{1\le j\le s}$
subordinate to the open cover $\{V_j\}_{1\le j\le s}$ of the
compact set $\overline{\Omega\cup \bigcup_{\eps\in (0,\eps_0]}
\Omega_\eps}$, see \cite[Page 12]{FeLa}.

We also define the deformation
$\Psi_{\eps,j}:\overline{\Omega_\eps \cap V_j}\to \overline{\Omega
\cap V_j}$ by $\Psi_{\eps,j}:=r_j^{-1}\circ \Phi_{\eps,j}\circ
r_j$. In this way $\Psi_{\eps,j}$ becomes a $C^{0,1}$
homeomorphism from $\overline{\Omega_\eps \cap V_j}$ onto
$\overline{\Omega \cap V_j}$ for any $j\in \{1,\dots,s\}$.

From the definition of $h_{\eps,j}$ and the restriction $\hat k>4$, we deduce that
\begin{equation} \label{eq:stima-Jacob}
\frac 12 \le {\rm det}(D\Psi_{\eps,j}(x))\le \frac 32 \qquad \text{for any } x\in \Omega_\eps\cap V_j \, .
\end{equation}
In order to show this, we observe that ${\rm det}(D\Psi_{\eps,j})={\rm det}(D\Phi_{\eps,j})=1-\frac{\partial h_{\eps,j}}{\partial x_N}(x',x_N)$.

Given $u\in H^1(\Omega)$ we put $u_j=\psi_j u$ for any $j\in
\{1,\dots,s\}$ in such a way that $u=\sum_{j=1}^s u_j$. Then we
define
\begin{equation} \label{eq:E-eps}
E_\eps u:=\sum_{j=1}^{s'} \tilde u_{\eps,j}+\sum_{j=s'+1}^s u_j
\in H^1(\Omega_\eps) \,
\end{equation}
where
\begin{equation} \label{eq:uj-tilde}
\tilde u_{\eps,j}(x)=
\begin{cases} u_j(\Psi_{\eps,j}(x)), & \qquad \text{if } x\in \Omega_\eps\cap V_j \, , \\[8pt]
 0, & \qquad \text{if } x\in \Omega_\eps\setminus V_j \, .
\end{cases}
\end{equation}
for any $j\in \{1,\dots,s'\}$.

\begin{remark} \label{r:cuboid} If $\Omega$ and $\Omega_\eps$ are in the form
\begin{align*}
& \Omega=\{(x',x_N)\in \R^N:x'\in W \ \text{and} \ a_N<x_N<g(x')\} \, , \\
& \Omega_\eps=\{(x',x_N)\in \R^N:x'\in W \ \text{and} \ a_N<x_N<g_\eps(x')\} \, ,
\end{align*}
where $W$ is a cuboid or a bounded domain in $\R^{N-1}$ of class
$C^{0,1}$ then the operator $E_\eps$ can be defined in the
following simple way
\begin{equation*}
E_\eps u(x)=u(\Phi_\eps(x)) \qquad \text{for any } x\in \Omega_\eps
\end{equation*}
where $\Phi_\eps(x',x_N)=(x',x_N-h_\eps(x',x_N))$ and $h_\eps$ is defined by \eqref{eq:def-h} with $g_\eps$ and $g$ in place of $g_{\eps,j}$ and $g_j$ respectively.
\end{remark}

We will show that the family of operators $\{E_\eps\}_{0<\eps\le
\eps_0}$ is really a connecting system in the sense of \cite{Vainikko}.
We first introduce some notations and state a preliminary result.

For any $y\in \Psi_{\eps,j}(\partial
\Omega_\eps\cap V_i \cap V_j)$ we put
$\Theta_{\eps,i,j}(y):=\Psi_{\eps,i}(\Psi_{\eps,j}^{-1}(y))$ in
order to define
\begin{equation}  \label{eq:Thetaij}
\Theta_{\eps,i,j}:\Psi_{\eps,j}(\partial\Omega_\eps\cap V_i\cap
V_j)\to \Theta_{\eps,i,j}(\Psi_{\eps,j}(\partial\Omega_\eps\cap
V_i\cap V_j))
\end{equation}
as a diffeomorphism between two open subsets of the manifold $\partial\Omega$.

For any $j\in \{1,\dots,s'\}$ let $\Gamma_j:\partial\Omega\cap
V_j\to W_j\subset \R^{N-1}$ be the maps defined by
\begin{equation} \label{eq:Gamma-j}
\Gamma_j(y):=P(r_j(y)) \qquad \text{for any } j\in \{1,\dots,s'\}
\end{equation}
where $P:\R^{N}\to \R^{N-1}$ is the projection $(x',x_N)\mapsto
x'$.
We observe that $\Gamma_j^{-1}:W_j \to
\partial\Omega\cap V_j$ satisfies
$\Gamma_j^{-1}(z')=r_j^{-1}((z',g_j(z')))$ for any $z'\in W_j$.

We now report below a result taken from \cite[Lemma 7]{FeLa} which was stated, in that setting, with a $C^{1,1}$ regularity assumption on the domains $\Omega_\eps$ but we observe that the arguments used in its proof work with a $C^{0,1}$ regularity assumption and our condition \eqref{eq:assumptions} as well.

\begin{proposition} \label{p:7} \ \cite[Lemma 7]{FeLa} \ Let $\mathcal A$ be an atlas. Let
$\{\Omega_\eps\}_{0<\eps\le\eps_0}$ be a family of domains of
class $C^{0,1}(\mathcal A)$ and $\Omega$ a domain of class
$C^{0,1}(\mathcal A)$. Assume the validity of condition \eqref{eq:assumptions}.

Let $\{\omega_\eps\}_{0<\eps\le \eps_0}\subset
L^2(\partial\Omega)$ be such that ${\rm supp}(\omega_\eps)\subset
\partial\Omega \cap V_i$ for any $\eps\in (0,\eps_0]$, for some $i\in
\{1,\dots,s'\}$. Suppose that there exists $\omega\in
L^2(\partial\Omega)$ such that $\omega_\eps\to \omega$ in
$L^2(\partial\Omega)$ as $\eps\to 0$. For $j\in \{1,\dots,s'\}$
let $\Theta_{\eps,i,j}$ be as in \eqref{eq:Thetaij}.
For any $\eps\in (0,\eps_0]$ define the function
\begin{equation*}
\tilde\omega_\eps(y):=
\begin{cases}
\omega_\eps(\Theta_{\eps,i,j}(y)) & \qquad \text{if } y\in \Psi_{\eps,j}(\partial\Omega_\eps\cap V_i\cap V_j)\, , \\
0 & \qquad \text{if } y\in \partial\Omega \setminus
\Psi_{\eps,j}(\partial\Omega_\eps\cap V_i\cap V_j)  \, .
\end{cases}
\end{equation*}
Then $\tilde\omega_\eps\to \omega\chi_{\partial\Omega\cap V_i\cap
V_j}$ in $L^2(\partial\Omega)$ as $\eps\to 0$.
\end{proposition}

We are ready to prove that the family of operators $\{E_\eps\}$ is a connecting system.

\begin{lemma} \label{l:1} Let $\mathcal A$ be an atlas. Let
$\{\Omega_\eps\}_{0<\eps\le\eps_0}$ be a family of domains of
class $C^{0,1}(\mathcal A)$, $\Omega$ a domain of class
$C^{0,1}(\mathcal A)$ and for any $\eps\in (0,\eps_0]$ let
$E_\eps$ be the map defined in \eqref{eq:E-eps}. Assume  the
validity of \eqref{eq:assumptions}-\eqref{eq:conv-per}.

Then the following assertions hold true:
\begin{itemize}
\item [(i)] the map $E_\eps:H^1(\Omega)\to H^1(\Omega_\eps)$ is continuous for any $\eps\in (0,\eps_0]$;

\item [(ii)] the family of operators $\{E_\eps\}_{0<\eps\le
\eps_0}$ satisfies \eqref{eq:conv-norm}.
\end{itemize}
\end{lemma}

\begin{proof} This lemma is essentially an adaptation of \cite[Lemma 2]{FeLa} with the obvious
changes. For this reason we only give an idea of the proof quoting
the necessary references contained in the proof of \cite[Lemma
2]{FeLa}.

Since the proof of (i) easily follows from the definition of
$E_\eps$, it is left to the reader.

It remains to show the validity of \eqref{eq:conv-norm}. Let $u\in
H^1(\Omega)$ and let $\tilde u_{\eps,j}$ be the functions
introduced in \eqref{eq:uj-tilde}.

Note that in order to prove \eqref{eq:conv-norm}, it is sufficient to prove the following convergences:
\begin{align} \label{eq:pas-lim}
\lim_{\eps\to 0} \int_{\Omega_\eps} \frac{\partial \tilde u_{\eps,i}}{\partial x_k}
\frac{\partial\tilde u_{\eps,j}}{\partial x_k}\, dx=\int_\Omega \frac{\partial u_i}{\partial x_k}
\frac{\partial u_j}{\partial x_k} \, dx \quad \forall i,j\in \{1,\dots,s'\} \, , \forall\, k\in \{1,\dots,N\};
\end{align}

\begin{align}  \label{eq:pas-lim-bis}
\lim_{\eps\to 0} \int_{\Omega_\eps} \frac{\partial \tilde u_{\eps,i}}{\partial x_k}
\frac{\partial u_{j}}{\partial x_k}\, dx=\int_\Omega \frac{\partial u_i}{\partial x_k}
\frac{\partial u_j}{\partial x_k} \, dx \quad \forall i\in \{1,\dots,s'\}, \forall j\in \{s'+1,\dots,s\}, \, \forall k\in \{1,\dots,N\};
\end{align}

\begin{align} \label{eq:pas-lim-boundary}
\lim_{\eps\to 0} \int_{\partial\Omega_\eps} \tilde u_{\eps,i}\, \tilde u_{\eps,j} \, dS=\int_{\partial\Omega} u_i\,  u_j \, dS \quad
\forall i,j\in \{1,\dots,s'\} \, , \forall\, k\in \{1,\dots,N\};
\end{align}

\begin{align}  \label{eq:pas-lim-boundary-bis}
\lim_{\eps\to 0} \int_{\partial \Omega_\eps} \tilde u_{\eps,i} \, u_{j} \, dS=\int_{\partial \Omega} u_i \,  u_j \, dS
\quad \forall i\in \{1,\dots,s'\}, \forall j\in \{s'+1,\dots,s\}, \, \forall k\in \{1,\dots,N\} \, .
\end{align}
We only give some details of the proof of \eqref{eq:pas-lim} and
\eqref{eq:pas-lim-boundary} since \eqref{eq:pas-lim-bis} and
\eqref{eq:pas-lim-boundary-bis} can be proved in a completely
equivalent way, actually easier.

We first prove \eqref{eq:pas-lim} in the case $i=j$. For any $x\in
\Omega_\eps\cap V_i$ we have
\begin{equation} \label{eq:compute}
\frac{\partial \tilde u_{\eps,i}}{\partial x_k}(x)=\sum_{l=1}^N
\frac{\partial u_i}{\partial x_l}(\Psi_{\eps,i}(x))
\frac{\partial[(\Psi_{\eps,i}(x))_l]}{\partial x_k} \, .
\end{equation}
By \eqref{eq:assumptions} and the definitions of $\Psi_{\eps,i}$,
$\Phi_{\eps,i}$ and $h_{\eps,i}$ we infer
\begin{equation} \label{eq:estimate-2}
\left\|\frac{\partial[(\Psi_{\eps,i}(x))_l]}{\partial
x_k}\right\|_{L^\infty(V_i\cap \Omega_\eps)}=O(1) \ \ \text{as }
\eps\to 0\, , \qquad \text{for any } l,k\in \{1,\dots,N\} \, .
\end{equation}
Exploiting \eqref{eq:compute}-\eqref{eq:estimate-2},
\eqref{eq:stima-Jacob} and the fact that $|r_i(\Omega\cap
V_i)\setminus K_{\eps,i}|\to 0$ as $\eps\to 0$ one gets
\begin{equation} \label{eq:pezzo-1}
\lim_{\eps\to 0} \int_{(\Omega_\eps \cap V_i)\setminus r_i^{-1}(K_{\eps,i})} \left( \frac{\partial\tilde u_{\eps,i}}{\partial x_k}\right)^2 dx=0 \, .
\end{equation}
On the other hand
\begin{equation} \label{eq:pezzo-2}
\lim_{\eps\to 0}\int_{r_i^{-1}(K_{\eps,i})}  \left( \frac{\partial\tilde u_{\eps,i}}{\partial x_k}\right)^2 dx
=\lim_{\eps\to 0} \int_{r_i^{-1}(K_{\eps,i})}  \left( \frac{\partial u_{i}}{\partial x_k}\right)^2 dx
=\int_{\Omega\cap V_i}  \left( \frac{\partial u_{i}}{\partial x_k}\right)^2 dx=\int_{\Omega}  \left( \frac{\partial u_{i}}{\partial x_k}\right)^2 dx .
\end{equation}
Combining \eqref{eq:pezzo-1} and \eqref{eq:pezzo-2} we conclude the proof of \eqref{eq:pas-lim} in the case $i=j$.

In order to prove \eqref{eq:pas-lim} in the general case $i\neq j$
one can use a suitable decomposition of the domains of
integrations like in \cite[Identity (39)]{FeLa}.

Let us proceed with the proof of \eqref{eq:pas-lim-boundary}.
Combining \eqref{eq:assumptions}-\eqref{eq:conv-per} (and hence
\eqref{eq:Jac}) with Proposition \ref{p:7}, we
obtain
\begin{align*}
& \lim_{\eps\to 0} \int_{\partial\Omega_\eps} \tilde u_{\eps,i}\, \tilde u_{\eps,j} \, dS
=\lim_{\eps\to 0} \int_{\partial\Omega_\eps\cap V_i \cap V_j} u_i(\Psi_{\eps,i}(x)) \, u_j(\Psi_{\eps,j}(x))  \, dS \\
& \qquad =\lim_{\eps\to 0} \int_{\Psi_{\eps,j}(\partial\Omega_\eps\cap V_i \cap V_j)} u_i(\Psi_{\eps,i}(\Psi_{\eps,j}^{-1}(y)))
\, u_j(y)  W_{\eps,j}(y) \, dS \\
& \qquad =\int_{\partial\Omega \cap V_i \cap V_j} u_i(y) \, u_j(y) \, dS=\int_{\partial\Omega} u_i \, u_j \, dS
\end{align*}
where we used the change of variables $y=\Psi_{\eps,j}(x)$ and
where we put
\begin{equation} \label{eq:W-eps-j}
W_{\eps,j}(y):=\frac{\sqrt{1+|\nabla_{x'}
g_{\eps,j}(\Gamma_j(y))|^2}}{\sqrt{1+|\nabla_{x'}
g_{j}(\Gamma_j(y))|^2}} \, .
\end{equation}
Here $\Gamma_j:\partial\Omega\cap V_j\to W_j\subset \R^{N-1}$ are
the maps defined in \eqref{eq:Gamma-j}.

This completes the proof of the lemma.
\end{proof}

\subsection{Minimax characterization for the eigenvalues of
\eqref{eq:Steklov}}  \label{ss:minimax}

We denote by $\{u_n\}_{n\ge 0}$ an orthonormal system of
eigenfunctions with respect to the scalar product of
$L^2(\partial\Omega)$, i.e. $\int_{\partial\Omega} u_n u_m \,
dS=\delta_{nm}$, where $u_n$ is an eigenfunction of $\lambda_n$
for any $n\ge 0$. In particular we have
\begin{equation} \label{eq:double-ortho-1}
(u_n,u_m)_0=\int_\Omega \nabla u_n \nabla u_m \, dx=\int_{\partial\Omega} u_n u_m \, dS=0
\qquad \text{for any } n,m\ge 0, \ n\neq m
\end{equation}
and
\begin{equation} \label{eq:double-ortho-2}
\int_\Omega |\nabla u_n|^2 dx=\lambda_n \, , \qquad   \int_{\partial\Omega} u_n^2 \, dS=1
\qquad \text{for any } n\ge 0 \, .
\end{equation}

Let $S:H^1(\Omega)\to H^1(\Omega)$ be the resolvent operator defined in Section \ref{s:functional-setting-B}.
We observe that $S$ is an operator which acts from $(H^1_0(\Omega))^\perp$ to itself, i.e.
$$
S:(H^1_0(\Omega))^\perp \to (H^1_0(\Omega))^\perp
$$
where orthogonality is meant with respect to the scalar product $(\cdot,\cdot)_0$.

Indeed, if we choose $w \in (H^1_0(\Omega))^\perp$, letting
$u=Sw$, we have $Tu=Jw$, i.e.
$$
_{(H^1(\Omega))'} \langle
Tu,v\rangle_{H^1(\Omega)}=_{(H^1(\Omega))'} \langle
Jw,v\rangle_{H^1(\Omega)} \qquad \text{for any } v\in H^1(\Omega)
\, .
$$
This is equivalent to
\begin{equation} \label{eq:id-sopra}
(u,v)_0=\int_{\partial\Omega} wv \, dS \qquad \text{for any } v\in
H^1(\Omega) \, .
\end{equation}
In particular, if we choose $v\in H^1_0(\Omega)$ in
\eqref{eq:id-sopra}, we obtain that the right hand side of
\eqref{eq:id-sopra} vanishes and so it does its left hand side,
thus proving that $u\in (H^1_0(\Omega))^\perp$.

The invariance property of the space $(H^1_0(\Omega))^\perp$ under
the action of the self-adjoint compact operator $S$ shows that the
basis $\{u_n\}_{n\ge 0}$ defined before, is an orthogonal basis of
$(H^1_0(\Omega))^\perp$.

A standard application of the classical minimax principle for eigenvalues allows to show that, for the
eigenvalues of \eqref{eq:Steklov}, the following characterization holds true:
\begin{equation} \label{eq:char-1}
\lambda_n=\min_{W\in \mathcal W_n} \  \max_{v\in W\setminus \{0\}}
\frac{\int_\Omega |\nabla v|^2 dx}{\int_{\partial \Omega} v^2 dS}
\end{equation}
where for any $n\ge 0$
\begin{equation} \label{eq:def-Wn}
\mathcal W_n:=\{W\subseteq (H^1_0(\Omega))^\perp \ \text{subspace}:{\rm dim}(W)=n+1\} \, .
\end{equation}

Actually, if $w\in (H^1_0(\Omega))^\perp$ and $v\in H^1_0(\Omega)$, not only we have $(w,v)_0=0$ by definition of orthogonality, but we also have
\begin{equation} \label{eq:ortho}
\int_\Omega \nabla w \nabla v \, dx=\int_{\partial\Omega} wv \, dS=0 \, .
\end{equation}

We also prove that  for any $n\ge 0$ the $n$-th eigenvalue
$\lambda_n$ admits the alternative $\inf$-$\sup$ characterization
\begin{align} \label{eq:alt-char}
\lambda_n=\inf_{V \in \mathcal V_n} \ \ \sup_{v\in V\setminus
H^1_0(\Omega)}\frac{\int_\Omega |\nabla v|^2
dx}{\int_{\partial\Omega} v^2 dS}
\end{align}
where
\begin{equation*}
\mathcal V_n:=\{V\subseteq H^1(\Omega) \ \text{subspace}: {\rm dim}(V)=n+1 \ \text{and} \ V\nsubseteq H^1_0(\Omega)\} \, .
\end{equation*}

\begin{proposition} \label{p:min-max}
The eigenvalues of \eqref{eq:Steklov} admit the variational characterization \eqref{eq:alt-char}.

Moreover we have:

\begin{itemize}
\item[(i)] if  $V\in \mathcal V_n$ satisfies $V\cap
H^1_0(\Omega)\neq \{0\}$ then
\begin{equation} \label{eq:supremum}
\sup_{v\in V\setminus H^1_0(\Omega)}\frac{\int_\Omega |\nabla v|^2
dx}{\int_{\partial\Omega} v^2 dS}=+\infty \, ;
\end{equation}

\item[(ii)] if $V\in \mathcal V_n$ satisfies $V\cap H^1_0(\Omega)=
\{0\}$ then the supremum in \eqref{eq:supremum} is finite and it
is achieved;

\medskip

\item[(iii)] the infimum in \eqref{eq:alt-char} is achieved so
that we may write
\begin{align} \label{eq:alt-char-bis}
\lambda_n=\min_{V \in \mathcal V_n} \ \sup_{v\in V\setminus
H^1_0(\Omega)}\frac{\int_\Omega |\nabla v|^2
dx}{\int_{\partial\Omega} v^2 dS} \, .
\end{align}
\end{itemize}
\end{proposition}

\begin{proof} We prove the three parts of the lemma separately.

{\bf Proof of (i).} Let $V$ be as in (i) and let $v\in V\setminus \left[H^1_0(\Omega)\cup (H^1_0(\Omega))^\perp\right]$.

Consider its orthogonal decomposition $v=v_0+v_1\in H^1_0(\Omega) \oplus (H^1_0(\Omega))^\perp$.
Since $v\not\in H^1_0(\Omega)\cup (H^1_0(\Omega))^\perp$ we clearly have that $v_0,v_1\neq 0$.
Let us use $v_t=tv_0+v_1$, $t\in (0,+\infty)$,
as a test function in the Rayleigh quotient appearing in \eqref{eq:supremum}.
By \eqref{eq:ortho} and the fact that $v_0$ has null trace on $\partial\Omega$, we have
\begin{align*}
\frac{\int_\Omega |\nabla v_t|^2 dx}{\int_{\partial\Omega} v_t^2 dS}
=\frac{t^2 \int_{\Omega}|\nabla v_0|^2 dx+\int_{\Omega}|\nabla v_1|^2 dx}{\int_{\partial\Omega} v_1^2 dS}\to +\infty
\qquad \text{as } t\to +\infty \, .
\end{align*}
This completes the proof of (i).

{\bf Proof of (ii).} First of all, if $V$ is as in (ii) we have that $V\setminus H^1_0(\Omega)=V\setminus\{0\}$. Due to the homogeneity property of the Rayleigh quotient we clearly have that
\begin{align*}
\sup_{v\in V\setminus H^1_0(\Omega)}\frac{\int_\Omega |\nabla v|^2 dx}{\int_{\partial\Omega} v^2 dS}
=\sup_{v\in V\setminus \{0\}} \frac{\int_\Omega |\nabla v|^2 dx}{\int_{\partial\Omega} v^2 dS}
=\sup_{v\in V, \|v\|_0=1}\frac{\int_\Omega |\nabla v|^2 dx}{\int_{\partial\Omega} v^2 dS}
=\max_{v\in V, \|v\|_0=1}\frac{\int_\Omega |\nabla v|^2 dx}{\int_{\partial\Omega} v^2 dS}
\end{align*}
where the last equality follows from the compactness of the unit
sphere in a finite dimensional space.

{\bf Proof of (iii).} Let $I(\mathcal W_n)$ be the minimax value
introduced in \eqref{eq:char-1} and let $I(\mathcal V_n)$ be the
inf-sup value defined in \eqref{eq:alt-char}. By (i)-(ii) we
clearly have that
\begin{equation} \label{eq:VinVn}
I(\mathcal W_n)\ge I(\mathcal V_n)=\inf_{V \in \mathcal V_n, V\cap H^1_0(\Omega)=\{0\}} \ \
\sup_{v\in V\setminus \{0\}}\frac{\int_\Omega |\nabla v|^2
dx}{\int_{\partial\Omega} v^2 dS}
\end{equation}
where the inequality above follows from the fact that $\mathcal
W_n\subset \mathcal V_n$.

On the other hand, for any $V\in \mathcal V_n$ with $V\cap
H^1_0(\Omega)=\{0\}$, let us consider the orthogonal projections
$P:V\to P(V)\subset H^1_0(\Omega)$ and $Q:V\to Q(V)\subset
(H^1_0(\Omega))^\perp$ so that
\begin{equation} \label{eq:decomposition}
v=Pv+Qv  \in H^1_0(\Omega)\oplus (H^1_0(\Omega))^\perp  \qquad \text{for any } v\in V \, .
\end{equation}
We claim that ${\rm dim}(Q(V))={\rm dim}(V)=n+1$. By definition we
have that the linear map $Q$ is surjective. Let us prove that it
is also injective. By contradiction suppose that $Q$ is not
injective so that, being $Q$ linear, there exists $v\in V\setminus
\{0\}$ such that $Qv=0$ which inserted in \eqref{eq:decomposition}
shows that $v=Pv\in H^1_0(\Omega)$ and this is in contradiction
with the fact that $V\cap H^1_0(\Omega)=\{0\}$. Being $Q$ an
isomorphism between vector spaces the proof of the claim follows.

For any $v\in V\setminus \{0\}$ with $V\cap H^1_0(\Omega)=\{0\}$,
let us put $v_0=Pv$ and $v_1=Qv$ in such a way that $v_1\neq 0$.
With this choice of $V$ and $v$, by \eqref{eq:ortho} we obtain
\begin{align*}
\frac{\int_\Omega |\nabla v|^2 dx}{\int_{\partial\Omega} v^2 dS}
=\frac{\int_\Omega |\nabla v_0|^2 dx+\int_\Omega |\nabla v_1|^2 dx}{\int_{\partial\Omega} v_1^2 dS}
\ge \frac{\int_\Omega |\nabla v_1|^2 dx}{\int_{\partial\Omega} v_1^2 dS}
\end{align*}
so that
\begin{align} \label{eq:ultimopasso}
\max_{v \in V\setminus\{0\}}\frac{\int_\Omega |\nabla v|^2 dx}{\int_{\partial\Omega} v^2 dS}
\ge \max_{w \in Q(V)\setminus\{0\}}\frac{\int_\Omega |\nabla w|^2 dx}{\int_{\partial\Omega} w^2 dS}\ge I(\mathcal W_n)
\end{align}
since $Q(V)\in \mathcal W_n$ being $Q(V)\subset (H^1_0(\Omega))^\perp$ and ${\rm dim}(Q(V))=n+1$.

Since \eqref{eq:ultimopasso} holds for any $V\in \mathcal V_n$
such that $V\cap H^1_0(\Omega)=\{0\}$, taking the infimum over $V$
in \eqref{eq:ultimopasso}, by the equality in the right hands side
of \eqref{eq:VinVn}, we conclude that $I(\mathcal V_n)\ge
I(\mathcal W_n)$. This combined with the inequality in left hand
side of \eqref{eq:VinVn} and with \eqref{eq:char-1} proves that
$I(\mathcal V_n)=I(\mathcal W_n)=\lambda_n$.

Finally, the fact that the infimum in \eqref{eq:alt-char} is
achieved follows from $I(\mathcal V_n)=\lambda_n$ combined with
the particular choice $V={\rm span}\{u_0,\dots,u_n\}\in \mathcal
V_n$: indeed, in this way the inequality
\begin{align*}  
\max_{v\in {\rm span}\{u_0,\dots,u_n\}\setminus\{0\}}
\ \frac{\int_\Omega |\nabla v|^2 dx}{\int_{\partial\Omega} v^2 dS}\le \lambda_n
\end{align*}
becomes an equality if one chooses $v=u_n$ and, in turn, the
above maximum achieves the minimum in \eqref{eq:alt-char}.
\end{proof}

\section{Main results} \label{s:main-results}


We start with the following  result on spectral convergence for problem \eqref{eq:Steklov}:

\begin{theorem} \label{t:main-1}
Let $\mathcal A$ be an atlas. Let $\{\Omega_\eps\}_{0<\eps\le
\eps_0}$ be a family of domains of class $C^{0,1}(\mathcal A)$ and
let $\Omega$ be a domain of class $C^{0,1}(\mathcal A)$. Assume
the validity of conditions
\eqref{eq:assumptions}-\eqref{eq:conv-per}. Then $S_{\eps} \CC  S$
with respect to the operators $E_{\eps}$ defined in
\eqref{eq:E-eps}. In particular, the spectrum of
\eqref{eq:Steklov} behaves continuously at $\eps=0$ in the sense
of Theorem \ref{vaithm}.
\end{theorem}

In the next theorem we relax assumptions of Theorem \ref{t:main-1}
by replacing condition \eqref{eq:conv-per} with
\begin{equation} \label{eq:weak-L1}
\sqrt{1+|\nabla_{x'}g_{\eps,j}|^2}\rightharpoonup \gamma_j \qquad
\text{weakly in } L^1(W_j) \quad \text{as } \eps \to 0
\end{equation}
where $\gamma_j\in L^1(W_j)$ for any $j=1,\dots,s'$.

Looking at \eqref{eq:weak-L1}, it seems reasonable to define the
function $\gamma:\partial\Omega\to \R$, locally given by
\begin{equation} \label{eq:gamma}
\gamma(y)=\frac{\gamma_j(\Gamma_j(y))}{\sqrt{1+|\nabla_{x'}g_{j}(\Gamma_j(y))|^2}}
\qquad \text{for any } y\in V_j\cap
\partial\Omega
\end{equation}
where $\Gamma_j$ is the map defined in \eqref{eq:Gamma-j}, for any
$j=1,\dots,s'$.

We observe that, arguing as in \cite[Lemma 5.1, Corollary
5.1]{ArBru}, one can deduce that the function $\gamma$ is well
defined being its definition not depending on local charts.

We now introduce the following weighted Steklov problem
\begin{equation} \label{eq:Steklov-weighted}
\begin{cases}
\Delta u=0, & \qquad \text{in } \Omega \, , \\
u_{\nu }=\lambda \, \gamma(x) u, & \qquad \text{on }
\partial\Omega \, .
\end{cases}
\end{equation}
Let us denote by
\begin{equation} \label{eq:spectrum-weighted}
0=\lambda_0(\gamma)<\lambda_1(\gamma)\le \lambda_2(\gamma) \le
\dots \le \lambda_n(\gamma)\le \dots
\end{equation}
the eigenvalues of \eqref{eq:Steklov-weighted}.

We now define on the space $H^1(\Omega)$ the scalar product
\begin{equation} \label{eq:prod-gamma}
(u,v)_{\gamma}:=\int_\Omega \nabla u \nabla v \,
dx+\int_{\partial\Omega} \gamma \, uv\, dS \qquad \text{for any }
u,v\in H^1(\Omega)
\end{equation}
and the corresponding norm
\begin{equation} \label{eq:norm-gamma}
\|u\|_{\gamma}:=(u,u)_{\gamma}^{1/2} \qquad \text{for any } u\in
H^1(\Omega) \, .
\end{equation}

We observe that the boundary integral in \eqref{eq:prod-gamma} is
well defined since by \eqref{eq:assumptions} and
\eqref{eq:weak-L1} we deduce that $\gamma\in L^{\infty}(\partial\Omega)$.

We observe that the norm \eqref{eq:norm-gamma} is equivalent to
the usual norm of $H^1(\Omega)$: indeed, for one estimate we can combine the fact that $\gamma\in L^\infty(\partial\Omega)$
with the classical trace inequality and for the
other one Lemma 2.7 (ii) and the fact that $\gamma\ge 1$, as one
can verify looking at \cite[Corollary 5.1]{ArBru}.

We now construct the operators $T_\gamma$, $J_\gamma$ simply by
replacing the boundary integrals in \eqref{eq:def-T} and
\eqref{eq:def-J} by $\int_{\partial\Omega} \gamma \, uv \, dS$.
The operator $S_\gamma$ is again defined by $T_\gamma^{-1}\circ
J_\gamma$.

The family of operators $E_\eps:H^1(\Omega)\to H^1(\Omega_\eps)$
can be defined exactly as in Section \ref{subsecoperatorsE}.

It can be proved that
\begin{equation} \label{eq:conv-norm-weighted}
\|E_\eps u\|_\eps \to \|u\|_\gamma \qquad \text{as } \eps\to 0 \,
, \qquad \text{for any } u \in H^1(\Omega)
\end{equation}
thus showing that the family of operators $\{E_\eps\}$ is still a
connecting system in the sense of Section
\ref{ss:convergenceint} provided that $H^1(\Omega)$ is endowed
with the norm $\|\cdot\|_\gamma$.

The proof of \eqref{eq:conv-norm-weighted} may be obtained by
proceeding exactly as in the proof of Lemma \ref{l:1} with the
only difference that in the concluding part of the proof of the
lemma, the function $W_{\eps,j}$ weakly converges in
$L^p(\partial\Omega\cap V_i \cap V_j)$ to the function
$\frac{\gamma_j}{\sqrt{1+|\nabla_{x'}g_\eps|^2}}\circ \Gamma_j$
for any $1\le p<\infty$, as a consequence of \cite[Remark 2.1]{ArBru}.

We are ready to state the following result.

\begin{theorem} \label{t:main-3}
Let $\mathcal A$ be an atlas. Let $\{\Omega_\eps\}_{0<\eps\le
\eps_0}$ be a family of domains of class $C^{0,1}(\mathcal A)$ and
let $\Omega$ be a domain of class $C^{0,1}(\mathcal A)$. Assume
the validity of conditions \eqref{eq:assumptions} and
\eqref{eq:weak-L1}.

Then $S_{\eps} \CC  S_\gamma$ with respect to the operators
$E_{\eps}$ defined in \eqref{eq:E-eps}. In particular, for any
$n\ge 1$ we have that
\begin{equation*}
\lambda_n^\eps\to \lambda_n(\gamma)  \qquad \text{as } \eps\to 0
\end{equation*}
where, according with the notation used in the Introduction, by
$\lambda_n^\eps$ we mean the eigenvalues of the Steklov problem in
$\Omega_\eps$ and by $\lambda_n(\gamma)$ the eigenvalues defined
in \eqref{eq:spectrum-weighted}.
\end{theorem}

We observe that Theorem \ref{t:main-3} becomes an instability
result whenever $\gamma\not\equiv 1$ on $\partial\Omega$.

We now state another instability result in which we consider a
family of domains $\{\Omega_\eps\}_{0<\eps\le \eps_0}$ of class
$C^{0,1}(\mathcal A)$ and a fixed domain of class
$C^{0,1}(\mathcal A)$. We assume that the corresponding functions
$g_{\eps,j}$ satisfy
\begin{equation} \label{eq:th-main-2-0}
\lim_{\eps\to 0} \|g_{\eps,j}-g_j\|_{L^\infty(W_j)}=0
\end{equation}
but differently from the statements of Theorem \ref{t:main-1} and Thorem \ref{t:main-3}, we assume the following blow up condition of the surface element:
\begin{equation}  \label{eq:th-main-2}
\lim_{\eps \to 0} \mathcal H^{N-1}\left(\left\{x'\in W_j:\sqrt{1+|\nabla_{x'} \, g_{\eps,j}(x')|^2}\le t\right\}\right)=0  \qquad \text{for any } t>0 \, ,
\end{equation}
where we have denoted by $\mathcal H^{N-1}$ the $(N-1)$-dimensional Lebesgue measure.

\begin{theorem} \label{t:main-2} Let $\mathcal A$ be an atlas. Let $\{\Omega_\eps\}_{0<\eps\le
\eps_0}$ be a family of domains of class $C^{0,1}(\mathcal A)$ and
let $\Omega$ be a domain of class $C^{0,1}(\mathcal A)$. Assume that \eqref{eq:th-main-2-0}
and \eqref{eq:th-main-2} hold true.
Then for any $n\ge 1$ we have that
\begin{equation*}
\lambda_n^\eps\to 0  \qquad \text{as } \eps\to 0
\end{equation*}
where according with the notation used in the Introduction, $\lambda_n^\eps$
denote the eigenvalues of the Steklov problem in $\Omega_\eps$.
\end{theorem}

In order to better compare the assumptions of Theorems
\ref{t:main-1}-\ref{t:main-2}, we consider a special case in which
every domain is covered by only one chart. Let $W$ be a cuboid or
a bounded domain in $\R^{N-1}$ of class $C^{0,1}$. Let us assume
that $\Omega_\eps$ is given by
\begin{equation} \label{eq:Omega-eps}
\Omega_\eps:=\{(x',x_N):x'\in W \, , -1<x_N<g_\eps(x')\}
\end{equation}
where $g_\eps(x')=\eps^\alpha b(x'/\eps)$ for any $x'\in W$ and the function $b$ satisfies:
\begin{equation}\label{eq:cond-b-Y}
  b\in C^{0,1}(\R^{N-1}), \quad b\ge 0 \ \ \text{in } \R^{N-1}, \quad b \ \text{is a $Y$-periodic function}
\end{equation}
where $Y=\left(-\frac 12,\frac 12\right)^{N-1}$ is the unit cell in
$\R^{N-1}$. We also assume that
\begin{equation} \label{eq:Lebesgue}
\mathcal H^{N-1}\left(\left\{x'\in \R^{N-1}:|\nabla_{x'} \, b(x')|=0\right\}\right)=0 \, .
\end{equation}
We also put $\Omega=W\times (-1,0)$.

The next result shows how the exponent $\alpha$ introduced in the
definition of $g_\eps$ plays a crucial role in determining the
validity of one of the three coupled conditions
\eqref{eq:assumptions} and \eqref{eq:conv-per},
\eqref{eq:assumptions} and \eqref{eq:weak-L1},
\eqref{eq:th-main-2-0} and \eqref{eq:th-main-2}.

\begin{proposition} \label{p:comparison}
Let $\{\Omega_\eps\}_{\eps\ge 0}$ be a family of domains like in
\eqref{eq:Omega-eps} with $g_\eps(x')=\eps^\alpha b(x'/\eps)$ if
$\eps>0$, $b$ satisfying \eqref{eq:cond-b-Y}-\eqref{eq:Lebesgue}, and with $\Omega_0=\Omega=W\times (-1,0)$. Then we have
\begin{itemize}
\item[(i)] if $\alpha>1$ then
\eqref{eq:assumptions}-\eqref{eq:conv-per} hold true with $W_j=W$,
$g_\eps$ in place of $g_{\eps,j}$ and $g_j\equiv 0$ , i.e.
\begin{align*}
& \lim_{\eps \to 0} \|g_{\eps}\|_{L^\infty(W)}=0 \, , \qquad
\left\|\frac{\partial g_\eps}{\partial
x_k}\right\|_{L^\infty(W)}=O(1) \quad \text{as } \eps\to 0  \quad \text{for any } k\in \{1,\dots,N\} \, , \\[8pt]
& {\rm Per}(\Omega_\eps)\to {\rm Per}(\Omega) \quad \text{as }
\eps\to 0 \, ;
\end{align*}

\item[(ii)] if $\alpha=1$ then \eqref{eq:assumptions},
\eqref{eq:weak-L1} hold true with $W_j=W$, $g_\eps$ in place of
$g_{\eps,j}$, $g_j\equiv 0$ and with the constant function $\int_Y
\sqrt{1+|\nabla_{x'} b(y')|^2}dy'$ in place of $\gamma_j$, i.e.
\begin{align*}
& \lim_{\eps \to 0} \|g_{\eps}\|_{L^\infty(W)}=0 \, , \qquad
\left\|\frac{\partial g_\eps}{\partial
x_k}\right\|_{L^\infty(W)}=O(1) \quad \text{as } \eps\to 0  \quad \text{for any } k\in \{1,\dots,N\} \, , \\[8pt]
& \sqrt{1+|\nabla_{x'}g_\eps|^2} \rightharpoonup \int_Y
\sqrt{1+|\nabla_{x'} b(y')|^2}\, dy' \qquad \text{weakly in }
L^1(W) \quad \text{as } \eps\to 0 \, ;
\end{align*}

\item[(iii)] if \, $0<\alpha<1$  then
\eqref{eq:th-main-2-0}-\eqref{eq:th-main-2} hold true with
$W_j=W$, $g_\eps$ in place of $g_{\eps,j}$ and $g_j\equiv 0$, i.e.
\begin{align*}
& \lim_{\eps \to 0} \|g_{\eps}\|_{L^\infty(W)}=0 \, , \qquad
\lim_{\eps \to 0} \mathcal H^{N-1}\left(\left\{x'\in
W:\sqrt{1+|\nabla_{x'} \, g_{\eps}(x')|^2}\le t\right\}\right)=0
\quad \text{for any } t>0 \, .
\end{align*}

\end{itemize}
\end{proposition}

Looking at the statement of Proposition \ref{p:comparison} it
becomes clear that, at least in the particular case of the family
of domains defined in \eqref{eq:Omega-eps}, the three couples of
assumptions \eqref{eq:assumptions} and \eqref{eq:conv-per},
\eqref{eq:assumptions} and \eqref{eq:weak-L1},
\eqref{eq:th-main-2-0} and \eqref{eq:th-main-2}, becomes
complementary.

We set now
\begin{align} \label{eq:Gamma-Sigma}
& \Gamma_\eps:=\{(x',g_\eps(x')):x'\in W\} \, , \qquad \Gamma:=\{(x',0):x'\in W\} \, ,    \\[7pt]
& \notag \Sigma_\eps:=\partial\Omega_\eps\setminus \Gamma_\eps \,
, \qquad \Sigma:=\partial\Omega \setminus \Gamma \, .
\end{align}
For any $\eps\ge 0$ consider the following modified Steklov
problem:
\begin{equation} \label{eq:Steklov-modified}
\begin{cases}
\Delta u=0 \, , & \qquad \text{in } \Omega_\eps\, , \\
u=0 \, , & \qquad \text{on } \Sigma_\eps \, , \\
u_\nu=\lambda u \, , & \qquad \text{on } \Gamma_\eps \, ,
\end{cases}
\end{equation}
with the notation $\Omega_0:=\Omega$.

Let us denote by $\mu_n^\eps$ the eigenvalues of
\eqref{eq:Steklov-modified} for any $n\in \N\cup \{0\}$ and
$\eps\ge 0$.

We also consider the weighted problem

\begin{equation} \label{eq:Steklov-weighted-bis}
\begin{cases}
\Delta u=0 \, , & \qquad \text{in } \Omega \, , \\
u=0 \, , & \qquad \text{on } \Sigma \, , \\
u_\nu=\lambda C_b \, u \, , & \qquad \text{on } \Gamma \, ,
\end{cases}
\end{equation}
where $C_b=\int_Y \sqrt{1+|\nabla_{x'}b(y')|^2} \, dy'$.

It is clear that, if we denote by $\mu_n(b)$ the eigenvalues of
\eqref{eq:Steklov-weighted-bis} for any $n\in \N\cup \{0\}$, then
$\mu_n(b)=\frac{\mu_n^0}{C_b}$.

Combining the arguments used in the proofs of Theorems
\ref{t:main-1}-\ref{t:main-2} with the statement of Proposition
\ref{p:comparison}, we obtain a trichotomy result which better
emphasizes the complementarity of conditions assumed in the three
main theorems.

\begin{theorem} \label{t:tricotomia}
Let $\{\Omega_\eps\}_{\eps\ge 0}$ be a family of domains like in
\eqref{eq:Omega-eps} with $g_\eps(x')=\eps^\alpha b(x'/\eps)$ if
$\eps>0$, $b$ satisfying \eqref{eq:cond-b-Y}-\eqref{eq:Lebesgue}, and with $\Omega_0=\Omega=W\times (-1,0)$.

Let $\mu_n^\eps$ be the eigenvalues of \eqref{eq:Steklov-modified}
for any $\eps\ge 0$ and let $\mu_n(b)$ be the eigenvalues of
\eqref{eq:Steklov-weighted-bis}. Then the following statements
hold true:

\begin{itemize}

\item[(i)] if $\alpha>1$ then $\mu_n^\eps \to \mu_n^0$ as $\eps\to
0$;

\medskip

\item[(ii)] if $\alpha=1$ then $\mu_n^\eps \to \mu_n(b)$ as
$\eps\to 0$, i.e. $\ds{\mu_n^\eps\to}\, \frac{\mu_n^0}{C_b}$ as
$\eps\to 0$;

\medskip

\item[(iii)] if $0<\alpha<1$ then $\mu_n^\eps\to 0$ as $\eps\to
0$.

\end{itemize}

\end{theorem}

\section{Proof of Theorem \ref{t:main-1}} \label{s:p-t:main-1}

Inspired by \cite{FeLa}, we define a map which acts between the
spaces $H^1(\Omega)$, $H^1(\Omega_\eps)$ in a reversed way with
respect to $E_\eps$. For any $w\in H^1(\Omega_\eps)$ we put
\begin{equation} \label{eq:hat-w} \widehat w_{\eps,j}(x)=
\begin{cases} w_j(\Psi_{\eps,j}^{-1}(x)), & \qquad \text{if } x\in \Omega\cap V_j \\[8pt]
 0, & \qquad \text{if } x\in \Omega\setminus V_j \, .
\end{cases}
\end{equation}
for any $j\in \{1,\dots,s'\}$ and $w_j:= \psi_j w$ for any $j\in \{1,\dots,s\}$. We define
\begin{equation}\label{inversodiE}
B_\eps w:=\sum_{j=1}^{s'} \widehat w_{\eps,j}+\sum_{j=s'+1}^s w_j.
\end{equation}
 In this way we have constructed a map
$B_\eps:H^1(\Omega_\eps)\to H^1(\Omega)$.

Next we prove the following

\begin{lemma} \label{l:weak-convergence} Let $\mathcal A$ be an atlas. Let
$\{\Omega_\eps\}_{0<\eps\le\eps_0}$ be a family of domains of
class $C^{0,1}(\mathcal A)$ and $\Omega$ a domain of class
$C^{0,1}(\mathcal A)$. Assume the validity of conditions
\eqref{eq:assumptions}-\eqref{eq:conv-per}. Let $w_\eps \in
H^1(\Omega_{\eps})$ with $0<\eps\le \eps_0$, and $w\in
H^1(\Omega)$ be such that $w_\eps \EC w$. If we put
$u_\eps:=S_{\eps} w_\eps$ and $u:=S w$, then $B_\eps
u_\eps\rightharpoonup u$ in $H^1(\Omega)$ as $\eps\to 0$.
\end{lemma}

\begin{proof} We divide the proof of the lemma into several steps.
The argument used in this proof is essentially based on the one
presented in the proof of \cite[Lemma 8]{FeLa}.

{\bf Step 1.} In this step we prove that $\|u_\eps\|_\eps$ is
uniformly bounded with respect to $\eps\in (0,\eps_0]$.

Indeed
\begin{align}\label{duno}
\|u_\eps\|_\eps ^2 & =\int_{\partial\Omega_\eps} w_\eps u_\eps \,
dS \le \left(\int_{\partial\Omega_\eps} w_\eps^2 \,
dS\right)^{1/2} \left(\int_{\partial\Omega_\eps} u_\eps^2 \,
dS\right)^{1/2}\le \|w_\eps\|_\eps \, \|u_\eps\|_\eps
\end{align}
from which it follows that $\|u_\eps\|_\eps \le \|w_\eps\|_\eps$.
Now we observe that $\|w_\eps\|_\eps$ is uniformly bounded since
\begin{equation} \label{eq:bound-weps}
\|w_\eps\|_\eps\le \|w_\eps-E_\eps w\|_\eps+\|E_\eps w\|_\eps=O(1)
\qquad \text{as } \eps\to 0
\end{equation}
as one can deduce by Definition \ref{d:E-convergence-ArCaLo},
\eqref{eq:norm-convergence}, \eqref{eq:conv-norm} and Lemma
\ref{l:1} (ii).

{\bf Step 2.} We prove that $\{B_\eps u_\eps\}_{0<\eps\le \eps_0}$
is bounded in $H^1(\Omega)$ and, in particular, that it is weakly
convergent in $H^1(\Omega)$ along a sequence $\eps_n\downarrow 0$
as $n\to +\infty$. Indeed, by \eqref{eq:hat-w},
\eqref{inversodiE}, \eqref{eq:assumptions}, \eqref{eq:stima-Jacob}
\eqref{eq:estimate-2} and some computations, one can prove that,
up to shrink $\eps_0$ if necessary,
\begin{equation} \label{eq:b-nabla-vol}
\int_\Omega |\nabla(B_\eps u_\eps)|^2 \, dx\le C
\int_{\Omega_\eps} |\nabla u_\eps|^2 \, dx \qquad \text{as }
\eps\to 0
\end{equation}
for some constant $C$ independent of $\eps$. For the same reason
one can prove that, up to shrink $\eps_0$ if necessary,
\begin{equation} \label{eq:b-boundary}
\int_{\partial \Omega} (B_\eps u_\eps)^2 \, dS\le C \int_{\partial
\Omega_\eps} u_\eps^2 \, dS
\end{equation}
for some constant $C$ independent of $\eps$. Combining
\eqref{eq:b-nabla-vol}-\eqref{eq:b-boundary}, we deduce that there
exists a constant $C$ independent of $\eps$ such that $\|B_\eps
u_\eps\|_0\le C \|u_\eps\|_\eps$ for any $\eps$ small enough. The
boundedness of $\|B_\eps u_\eps\|_0$ now follows by Step 1.

Hence, we have that there
exists $\tilde u\in H^1(\Omega)$ such that, along a sequence $\eps_n \downarrow 0$, $B_{\eps_n} u_{\eps_n}
\rightharpoonup \tilde u$ in $H^1(\Omega)$. For simplicity in the
sequel we only write $B_\eps u_{\eps} \rightharpoonup \tilde u$ as
$\eps\to 0$ for denoting this convergence along the sequence
$\{\eps_n\}$. By passing to the limit in the following identity
\begin{equation} \label{eq:u_eps}
(u_\eps,E_\eps \varphi)_\eps=\int_{\partial\Omega_\eps} w_\eps \,
E_\eps \varphi \, dS \qquad \text{for any } \varphi\in H^1(\Omega)
\, ,
\end{equation}
in the next steps we will show that $\tilde u=u$.

{\bf Step 3.} In this step we pass to the limit in the left hand
side of \eqref{eq:u_eps}. We follow closely the argument contained
in the proof of Step 3 in \cite[Lemma 8]{FeLa}.

We define $K_\eps:=\left(\bigcup_{j=1}^{s'}
r_j^{-1}(K_{\eps,j})\right)\cup \left(\bigcup_{j=s'+1}^s
V_j\right)$ and we split the left hand side of \eqref{eq:u_eps} in
the following way
\begin{align} \label{eq:splitting-2}
(u_\eps,E_\eps \varphi)_\eps=Q_{K_\eps}(u_\eps,E_\eps
\varphi)+Q_{\Omega_\eps\setminus K_\eps}(u_\eps,E_\eps
\varphi)+\int_{\partial\Omega_\eps} u_\eps E_\eps \varphi \, dS
\end{align}
where $K_{\eps,j}$ denotes the set defined in Section
\ref{subsecoperatorsE} and, for any measurable set $A\subset
\R^N$, $Q_A(\cdot,\cdot)$ is the bilinear form defined by
\begin{equation} \label{eq:bilinear}
Q_A(u,v):=\int_A \nabla u \nabla v\, dx \, .
\end{equation}
We also denote by $Q_A(\cdot)$ the quadratic form
\begin{equation} \label{eq:quadratic}
Q_A(u):=Q_A(u,u)=\int_A |\nabla u|^2 dx \, .
\end{equation}

Following the argument employed for proving \cite[Estimate
(82)]{FeLa} with $Q_A$ replaced by our $Q_A$ defined in
\eqref{eq:bilinear} and \eqref{eq:quadratic}, we obtain
\begin{align} \label{eq:u-E-phi}
& Q_{K_\eps}(u_\eps,E_\eps
\varphi)=Q_{K_\eps}(u_\eps,\varphi)+o(1) \, .
\end{align}
Similarly one can prove that
\begin{align} \label{eq:B-u-phi}
  Q_{K_\eps}(B_\eps u_\eps,\varphi)=Q_{K_\eps}(u_\eps,\varphi)+o(1) \qquad \text{as } \eps\to 0 \, .
\end{align}
For more details about \eqref{eq:B-u-phi} see the proof of
\cite[Estimate (84)]{FeLa}.

Combining \eqref{eq:u-E-phi} and \eqref{eq:B-u-phi} we obtain
\begin{equation} \label{eq:scarico-E-B}
Q_{K_\eps}(u_\eps,E_\eps\varphi)=Q_{K_\eps}(B_\eps u_\eps,\varphi)+o(1) \qquad \text{as } \eps\to 0 \, .
\end{equation}
Now, proceeding as in \cite[Estimates (86)-(87)]{FeLa}, for the
second term on the right hand side of \eqref{eq:splitting-2}, we
have
\begin{equation} \label{eq:resto-1}
Q_{\Omega_\eps\setminus K_\eps}(u_\eps,E_\eps \varphi)=o(1) \qquad
\text{as } \eps\to 0 \, .
\end{equation}
Similarly we also have
\begin{equation} \label{eq:resto-1-0}
Q_{\Omega\setminus K_\eps}(B_\eps u_\eps,\varphi)=o(1) \qquad
\text{as } \eps \to 0 \, .
\end{equation}
Combining \eqref{eq:resto-1}-\eqref{eq:resto-1-0} with
\eqref{eq:scarico-E-B} we infer
\begin{equation}\label{eq:passaggio-01}
Q_{\Omega_\eps}(u_\eps,E_\eps\varphi)=Q_{K_\eps}(B_\eps u_\eps,\varphi)+Q_{\Omega\setminus K_\eps}(B_\eps u_\eps,\varphi)+o(1)
=Q_\Omega(B_\eps u_\eps,\varphi)+o(1) \qquad \text{as } \eps\to 0 \, .
\end{equation}
Let us consider now the third term in the right hand side of
\eqref{eq:splitting-2}. We proceed as follows:
\begin{align} \label{eq:bou-est-0}
& \int_{\partial\Omega_\eps} u_\eps E_\eps \varphi \,
dS=\sum_{i,j=1}^{s'} \int_{\partial\Omega_\eps\cap V_i\cap V_j}
u_{\eps,j}(x) \varphi_i(\Psi_{\eps,i}(x)) \, dS
\\
& \notag \qquad =\sum_{i,j=1}^{s'}
\int_{\Psi_{\eps,j}(\partial\Omega_\eps\cap V_i\cap V_j)} \widehat
u_{\eps,j}(y) \, \varphi_i(\Psi_{\eps,i}(\Psi_{\eps,j}^{-1}(y)))\,
W_{\eps,j}(y) \, dS
\end{align}
with $W_{\eps,j}$ as in \eqref{eq:W-eps-j}. By \eqref{eq:conv-per}
and its consequence \eqref{eq:Jac}, we deduce that the trivial
extension of $W_{\eps,j}$ to the whole $\partial\Omega$ converges
almost everywhere to the function $\chi_{\partial\Omega\cap V_j}$.

Therefore, by \eqref{eq:assumptions}, \eqref{eq:Jac} and
Proposition \ref{p:7}, we obtain as $\eps\to 0$
\begin{align} \label{eq:bou-est}
& \int_{\partial\Omega_\eps} u_\eps E_\eps \varphi \, dS
=\sum_{i,j=1}^{s'} \int_{\partial\Omega\cap V_i \cap V_j} \widehat
u_{\eps,j}(y) \, \varphi_i(y) \, dS+o(1) \\
& \notag \qquad =\sum_{i=1}^{s'} \int_{\partial\Omega\cap V_i}
B_\eps u_\eps \, \varphi_i \, dS+o(1)=\int_{\partial\Omega} B_\eps
u_\eps \, \varphi \, dS+o(1) \, .
\end{align}
Since $B_\eps u_\eps \rightharpoonup \tilde u$ in $H^1(\Omega)$,
inserting \eqref{eq:passaggio-01} and \eqref{eq:bou-est} into
\eqref{eq:splitting-2} and exploiting the continuity of the trace
map from $H^1(\Omega)$ into $L^2(\partial\Omega)$, we obtain
\begin{equation} \label{eq:passaggio-1}
(u_\eps,E_\eps\varphi)_\eps \to Q_\Omega(\tilde
u,\varphi)+\int_{\partial\Omega} \tilde u\varphi \, dS=(\tilde
u,\varphi)_0 \qquad \text{as } \eps\to 0 \, .
\end{equation}

{\bf Step 4.} The next purpose is to pass to the limit in the
right hand side of \eqref{eq:u_eps}.

First of all we observe that thanks to Lemma \ref{l:1} and the
fact that $w_\eps\EC w$
\begin{align} \label{eq:sost}
& \left|\int_{\partial\Omega_\eps} w_\eps\, E_\eps \varphi \,
dS-\int_{\partial\Omega_\eps} E_\eps w \, E_\eps \varphi \,
dS\right| \le  \left(\int_{\partial \Omega_\eps} |w_\eps-E_\eps
w|^2 \, dS\right)^{1/2}\cdot \left(\int_{\partial \Omega_\eps}
|E_\eps \varphi|^2 dS\right)^{1/2} \\[7pt]
& \notag \qquad \le \|w_\eps-E_\eps w\|_\eps\,
\|E_\eps\varphi\|_\eps=o(1) \qquad \text{as } \eps\to 0 \, .
\end{align}
Proceeding as for the proof of \eqref{eq:bou-est} one can prove
that
\begin{equation} \label{eq:claim}
\int_{\partial\Omega_\eps} E_\eps w \, E_\eps \varphi \, dS\to
\int_{\partial\Omega} w \varphi \, dS \qquad \text{as } \eps\to 0
\, .
\end{equation}
Combining \eqref{eq:sost} and \eqref{eq:claim}, we conclude that
\begin{equation} \label{eq:passaggio-2}
\int_{\partial\Omega_\eps} w_\eps \, E_\eps \varphi \, dS \to
\int_{\partial\Omega} w \varphi \, dS \qquad \text{as } \eps\to 0
\, .
\end{equation}

{\bf Step 5.} In this last step we complete the proof of the
lemma.

Inserting \eqref{eq:passaggio-1} and \eqref{eq:passaggio-2} into
\eqref{eq:u_eps} we deduce that
\begin{equation*}
(\tilde u,\varphi)_0=\int_{\partial\Omega} w \varphi \, dS \qquad
\text{for any } \varphi\in H^1(\Omega)\, .
\end{equation*}
We have shown that $\tilde u$ and $u$ are solutions of the same
variational problem which admits a unique solution as the reader
can easily check. This proves that $\tilde u=u$. In particular
this means that the weak limit $\tilde u$ does not depend on the
choice of the sequence $\eps_n \downarrow 0$, thus proving that
the convergence $B_\eps u_\eps\rightharpoonup u$ does not occur
only along a special sequence but as $\eps \to 0$ in the usual
sense. This completes the proof of the lemma.
\end{proof}

\begin{lemma} \label{l:E-conv} Let $\mathcal A$ be an atlas. Let
$\{\Omega_\eps\}_{0<\eps\le\eps_0}$ be a family of domains of
class $C^{0,1}(\mathcal A)$ and $\Omega$ a domain of class
$C^{0,1}(\mathcal A)$. Assume the validity of conditions
\eqref{eq:assumptions}, \eqref{eq:conv-per}.
 Let $w_\eps \in H^1(\Omega_{\eps})$ with $0<\eps\le \eps_0$, and $w\in H^1(\Omega)$ be
such that $w_\eps \EC w$. If we put $u_\eps:=S_\eps w_\eps$ and
$u:=S w$ then $u_\eps \EC u$. In particular this implies that
$S_\eps \EEC S$ as $\eps\to 0$ in the sense of Definition
\ref{d:EE-convergence-ArCaLo}.
\end{lemma}

\begin{proof} We use the notation for $Q_A(\cdot,\cdot)$ and $Q_A(\cdot)$ introduced in \eqref{eq:bilinear} and
\eqref{eq:quadratic}.

We write
\begin{equation} \label{eq:conv-finale}
\|u_\eps-E_\eps u\|_{\eps}^2=\|u_\eps\|_{\eps}^2-2(u_\eps,E_\eps
u)_\eps+\|E_\eps u\|_{\eps}^2 \, .
\end{equation}
By \eqref{eq:passaggio-01}, \eqref{eq:bou-est} and
\eqref{eq:passaggio-1} and the fact that $B_\eps u_\eps \rightharpoonup u$ in
$H^1(\Omega)$ as proved in Lemma \ref{l:weak-convergence}, we obtain
\begin{equation} \label{eq:conv-1}
(u_\eps,E_\eps u)_\eps=(B_\eps u_\eps,u)_0+o(1)\to (u,u)_0=\|u\|_0^2 \, .
\end{equation}
Moreover, by Lemma \ref{l:1} we have
\begin{equation} \label{eq:conv-2}
\|E_\eps u\|_{\eps}^2\to \|u\|_{0}^2 \, .
\end{equation}
We now prove that $\|u_\eps\|_{\eps}^2\to \|u\|_{0}^2$. We write
\begin{align} \label{eq:57}
& \|u_\eps\|_\eps^2=\int_{\partial\Omega_\eps} w_\eps \, u_\eps \,
dS=\int_{\partial\Omega_\eps} E_\eps w \, u_\eps \, dS+o(1)
\end{align}
where the second identity can be obtained by proceeding exactly as
in \eqref{eq:sost} and exploiting the fact that
$\|u_\eps\|_{\eps}=O(1)$ as $\eps\to 0$ as we have shown in Step 1
of Lemma \ref{l:weak-convergence}.

We claim that
\begin{equation} \label{eq:claim-2}
\int_{\partial\Omega_\eps} E_\eps w\,  u_\eps \, dS\to
\int_{\partial\Omega} w \, u \, dS=(u,u)_0=\|u\|_0^2 \, .
\end{equation}
Once \eqref{eq:claim-2} is proved, combining \eqref{eq:conv-1}-\eqref{eq:claim-2} with
\eqref{eq:conv-finale} the proof of the lemma follows.
Therefore, in order to complete the proof of the lemma, we only have to prove the validity of \eqref{eq:claim-2}.

In order to estimate $\int_{\partial\Omega_\eps} E_\eps w \,
u_\eps \, dS$ we proceed as follows:
\begin{align} \label{eq:computation}
& \int_{\partial\Omega_\eps} E_\eps w \, u_\eps \, dS
=\sum_{i,j=1}^{s'} \int_{\partial\Omega_\eps \cap V_i \cap V_j}
w_{i}(\Psi_{\eps,i}(x)) u_{\eps,j}(x) \, dS
\\
\notag & =\sum_{i,j=1}^{s'}
\int_{\Psi_{\eps,j}(\partial\Omega_\eps \cap V_i \cap V_j)} w_{i}
(\Psi_{\eps,i}(\Psi_{\eps,j}^{-1}(y))) \widehat u_{\eps,j}(y)
W_{\eps,j}(y) \, dS \, .
\end{align}
with $W_{\eps,j}$ as in \eqref{eq:W-eps-j}.

Applying Proposition \ref{p:7} to $w$, exploiting the fact that
$B_\eps u_\eps \rightharpoonup u$ in $H^1(\Omega)$ and recalling
that by \eqref{eq:conv-per}, the trivial extension of
$W_{\eps,j}$ to the whole $\partial\Omega$ converges almost
everywhere to the function $\chi_{\partial\Omega \cap V_j}$
as explained in Step 3 of the proof of Lemma
\ref{l:weak-convergence}, as $\eps\to 0$, we have
\begin{align*}
& \int_{\partial\Omega_\eps} E_\eps w \, u_\eps \, dS
=\sum_{i,j=1}^{s'} \int_{\partial\Omega\cap V_i\cap V_j}
w_{i} \widehat u_{\eps,j} \, dS+o(1)\\
& =\sum_{i=1}^{s'} \int_{\partial\Omega\cap V_i} w_{i} B_\eps
u_\eps \, dS+o(1)=\sum_{i=1}^{s'} \int_{\partial\Omega\cap V_i}
w_{i} u \, dS+o(1)=\int_{\partial\Omega} wu \, dS+o(1) \, .
\end{align*}
This completes the proof of \eqref{eq:claim-2}.
\end{proof}

\begin{lemma} \label{l:compact} Let $\mathcal A$ be an atlas. Let
$\{\Omega_\eps\}_{0<\eps\le\eps_0}$ be a family of domains of
class $C^{0,1}(\mathcal A)$ and $\Omega$ a domain of class
$C^{0,1}(\mathcal A)$. Assume the validity of conditions
\eqref{eq:assumptions}, \eqref{eq:conv-per}. Let $w_\eps \in H^1(\Omega_{\eps})$ with
$0<\eps\le \eps_0$ be such that $\|w_\eps\|_{\eps}=1$.
Then $\{S_{\eps} w_\eps\}_{0<\eps\le \eps_0}$ is precompact in
the sense of Definition \ref{d:precompact-ArCaLo}. In particular,
by Definition \ref{d:CC-convergence-ArCaLo} and Lemma
\ref{l:E-conv} we have that $S_{\eps} \CC S$.
\end{lemma}

\begin{proof} As in Lemma
\ref{l:weak-convergence} we put $u_\eps:=S_{\eps} w_\eps$.
Since $\|w_\eps\|_{\eps}=1$, proceeding as in the proof of Lemma \ref{l:weak-convergence},
one can show that $B_\eps u_\eps \rightharpoonup \widetilde u$ along
a sequence, for some $\widetilde u\in H^1(\Omega)$. We divide the
remaining part of the proof into four steps. We observe that, as in
the proof of Lemma \ref{l:weak-convergence}, $u_\eps$ satisfies \eqref{eq:u_eps} for any $\eps\in
(0,\eps_0]$.

{\bf Step 1.} In this step we pass to the limit in the right hand
side of \eqref{eq:u_eps}.

We proceed as follows:
\begin{align} \label{eq:computation-0-bis}
& \int_{\partial\Omega_\eps} w_\eps E_\eps \varphi \, dS=
\sum_{i,j=1}^{s'} \int_{\partial\Omega_\eps \cap V_i \cap V_j}
w_{\eps,i}(x) \varphi_{j}(\Psi_{\eps,j}(x)) \, dS
\\
\notag & =\sum_{i,j=1}^{s'}
\int_{\Psi_{\eps,j}(\partial\Omega_\eps \cap V_i \cap V_j)}
w_{\eps,i}(\Psi_{\eps,j}^{-1}(y)) \varphi_{j}(y) W_{\eps,j}(y) \,
dS \, .
\end{align}
with $W_{\eps,j}$ as in \eqref{eq:W-eps-j}. As in the proofs of Lemmas \ref{l:weak-convergence}-\ref{l:E-conv}, we have that the trivial extension of $W_{\eps,j}$ to
the whole $\partial\Omega$ converges almost everywhere to the
function $\chi_{\partial\Omega\cap V_j}$ and it remains
uniformly bounded as $\eps\to 0$ thanks to \eqref{eq:assumptions}.

Since $\|w_\eps\|_{\eps}=1$, by \eqref{eq:hat-w}
we
deduce that
$$
\|\widehat w_{\eps,i}\|_{0}=O(1) \qquad \text{as } \eps\to 0  \, , \ \text{for any } i\in \{1,\dots,s'\} \, ,
$$
see Step 2 in the proof of Lemma \ref{l:weak-convergence} for more details.

Hence, by the classical trace inequality for the space $H^1(\Omega)$, we also have
$$
\|\widehat w_{\eps,i}\|_{H^{1/2}(\partial\Omega)}=O(1) \qquad
\text{as } \eps\to 0  \ \text{for any } i\in \{1,\dots,s'\} \, .
$$
Since the embedding $H^{1/2}(\partial\Omega)\subset
L^2(\partial\Omega)$ is compact, $\{\widehat
w_{\eps,i}\}_{0<\eps\le \eps_0}$ is precompact in
$L^2(\partial\Omega)$.

Then, along a sequence $\eps_k\downarrow 0$,
we may assume that $\widehat w_{\eps_k,i}\to F_{i}$ in
$L^2(\partial\Omega)$ as $k\to +\infty$. For simplicity, in the
rest of the proof of the lemma we will omit the subindex $k$ and
we simply write $\widehat w_{\eps,i}\to F_{i}$ in
$L^2(\partial\Omega)$ as $\eps\to 0$.

Letting $\Theta_{\eps,i,j}$ be as in \eqref{eq:Thetaij}, we see that $\widehat
w_{\eps,i}(\Theta_{\eps,i,j}(y))=w_{\eps,i}(\Psi_{\eps,j}^{-1}(y))$
for any $y\in \Psi_{\eps,j}(\partial\Omega_\eps\cap V_i\cap V_j)$.

Then, applying Proposition \ref{p:7} to $\widehat w_{\eps,i}$, by
\eqref{eq:computation-0-bis} we obtain
\begin{align} \label{eq:PassLim-1}
\int_{\partial\Omega_\eps} w_\eps E_\eps \varphi \, dS\to
\sum_{i,j=1}^{s'} \int_{\partial\Omega\cap V_i\cap V_j} F_{i}
\varphi_j \, dS=\int_{\partial\Omega} F \, \varphi \, dS \qquad \text{as } \eps\to 0
 \end{align}
where we put $F=\sum_{i=1}^{s'} F_{i}$.

{\bf Step 2.} In this step we pass to the limit in the left hand
side of \eqref{eq:u_eps}.

One can proceed as in the proof of Step 3 in Lemma
\ref{l:weak-convergence}, where that argument was only based on the fact that $\|u_\eps\|_\eps=O(1)$ as $\eps\to 0$, as in the present case. Thus, \eqref{eq:passaggio-1} still holds true.

Combining \eqref{eq:passaggio-1}, \eqref{eq:PassLim-1} with
\eqref{eq:u_eps}, we infer
\begin{equation} \label{eq:F-nu}
(\widetilde u,\varphi)_0=\int_{\partial\Omega} F \varphi \,  dS \qquad \text{for any } \varphi \in H^1(\Omega)
\, .
\end{equation}

{\bf Step 3.} In this step we pass to the limit in
the right hand side of the following identity
\begin{equation} \label{eq:u-eps-bis}
(u_\eps,u_\eps)_\eps=\int_{\partial\Omega_\eps}
w_\eps u_\eps \, dS \, .
\end{equation}
As we did for \eqref{eq:computation-0-bis}, we write
\begin{align} \label{eq:conv-forte}
& \int_{\partial\Omega_\eps} w_\eps u_\eps \, dS
=\sum_{i,j=1}^{s'}
\int_{\Psi_{\eps,j}(\partial\Omega_\eps \cap V_i \cap V_j)}
w_{\eps,i}(\Psi_{\eps,j}^{-1}(y)) \widehat
u_{\eps,j}(y) W_{\eps,j}(y) \, dS
\end{align}
with $W_{\eps,j}$ as in Step 1, and
$\widehat u_{\eps,j}$ as in \eqref{eq:hat-w}.

Proceeding as in the proof of the validity of
\eqref{eq:b-nabla-vol} and \eqref{eq:b-boundary}, we deduce that $\|\widehat u_{\eps,j}\|_0=O(1)$ as $\eps\to 0$. Therefore,
passing to the limit along a subsequence $\{\eps_{k_n}\}$ of the sequence $\{\eps_k\}$ introduced in Step 1, for any $j\in \{1,\dots,s'\}$, there exists a function
$U_j\in H^1(\Omega)$ such that $\widehat u_{\eps_{k_n},j}\rightharpoonup
U_j$ in $H^1(\Omega)$ as $n\to +\infty$.

With the same notations of Step 1, we simply
write $\eps\to 0$ to denote the convergence along the subsequence $\{\eps_{k_n}\}$. By
\eqref{eq:conv-forte} and the compactness argument of Step 1, we then have
\begin{equation} \label{eq:conv-forte-2}
\int_{\partial\Omega_\eps} w_\eps u_\eps \, dS \to
\sum_{i,j=1}^{s'} \int_{\partial\Omega\cap V_i
\cap V_j} F_{i} U_j \, dS \qquad \text{as } \eps\to 0 \, .
\end{equation}
Since $B_\eps u_\eps=\sum_{j=1}^{s'} \widehat u_{\eps,j}$ on
$\partial\Omega$, from compactness of the trace map from $H^1(\Omega)$ into $L^2(\partial\Omega)$ and the uniqueness of the strong limit in
$L^2(\partial\Omega)$, one immediately obtains $\widetilde u=\sum_{j=1}^{s'}
U_j$ on $\partial\Omega$, which inserted into \eqref{eq:conv-forte-2} gives
\begin{equation}  \label{eq:conv-forte-3}
\int_{\partial\Omega_\eps} w_\eps u_\eps \, dS \to
 \int_{\partial\Omega} F \widetilde u\, dS  \qquad \text{as } \eps\to 0\, .
\end{equation}

{\bf Step 4.} In this step we conclude the proof of the lemma.

Choosing $\varphi=\widetilde u$ in \eqref{eq:F-nu} and combining
this with \eqref{eq:u-eps-bis}, \eqref{eq:conv-forte-3}, we obtain
as $\eps\to 0$ along an appropriate sequence
\begin{equation} \label{eq:conv-norm-1}
\|u_\eps\|_{\eps}^2=\int_{\partial\Omega_\eps}
w_\eps u_\eps \, dS\to \int_{\partial\Omega} F \widetilde u \, dS=\|\widetilde u\|_{0}^2 \, .
\end{equation}
On the other hand, by \eqref{eq:u_eps}, \eqref{eq:PassLim-1},
\eqref{eq:F-nu} with $\varphi=\widetilde u$, we obtain as $\eps\to
0$ along the same sequence converging to zero,
\begin{equation} \label{eq:conv-norm-2}
(u_\eps,E_\eps \widetilde u)_\eps=\int_{\partial\Omega_\eps} w_\eps E_\eps \widetilde u \, dS\to \int_{\partial\Omega} F\widetilde u \, dS
=\|\widetilde u\|_{0}^2 \, .
\end{equation}
Combining \eqref{eq:conv-norm-1} and \eqref{eq:conv-norm-2} with
Lemma \ref{l:1} (ii), we obtain
\begin{align*}
\|u_\eps-E_\eps \widetilde
u\|_{\eps}^2=\|u_\eps\|_{\eps}^2-2(u_\eps,E_\eps
\widetilde u)_\eps+\|E_\eps \widetilde u\|_{\eps}^2\to 0 \, .
\end{align*}
This proves that, along a sequence converging to zero, we have $u_\eps \EC \widetilde
u$ or equivalently that
$\{S_{\eps}w_\eps\}_{0<\eps\le \eps_0}$ is precompact in the
sense of Definition \ref{d:precompact-ArCaLo}. The proof of the lemma now follows from Lemma \ref{l:E-conv} and Definition \ref{d:CC-convergence-ArCaLo}.
\end{proof}

The proof of the theorem now follows combining Lemma \ref{l:compact}, which states the validity of the compact convergence $S_{\eps} \CC S$, with the abstract result stated in Theorem \ref{vaithm}.

As a bypass product of a number of results proved in this section, we have the following proposition
which we believe has its own interest  since it clarifies even more the meaning of $E$-convergence with respect to the operators  $E_{\eps}$ used above.

\begin{proposition}\label{intersec}
 Let $\mathcal A$ be an atlas. Let
$\{\Omega_\eps\}_{0<\eps\le\eps_0}$ be a family of domains of
class $C^{0,1}(\mathcal A)$, $\Omega$ a domain of class
$C^{0,1}(\mathcal A)$ and for any $\eps\in (0,\eps_0]$ let $E_\eps$ be the map defined in \eqref{eq:E-eps}. Assume  the validity of  condition  \eqref{eq:assumptions}, \eqref{eq:conv-per}.
If  $u_{\eps}\in H^1(\Omega_{\eps})$, $u\in H^1(\Omega)$ is  such that $u_{\eps}\EC u$ as $\eps \to 0 $ then
\begin{equation}\label{intersec0}
\| u_{\eps}- u\|_{H^1(\Omega_{\eps}\cap \Omega )}\to 0  \ \ {\rm and } \ \ \|u_{\eps}-E_{\eps} u   \|_{L^2(\partial \Omega_{\eps})}\to 0,
\end{equation}
as $\eps\to 0$.
\end{proposition}

\begin{proof}Let $K_{\eps}\subset \Omega_{\eps}\cap \Omega  $ be as in the proof of  Lemma~\ref{l:weak-convergence}. Recall  that
$E_{\eps}u=u$ on $ K_{\eps}$.

We focus the attention on the proof of the first convergence in \eqref{intersec0} since the second one is a trivial consequence of the definition of $\|\cdot\|_\eps$ and the definition of $E$-convergence, see \eqref{eq:equiv-norms} and Definition \ref{d:E-convergence-ArCaLo} respectively. In the rest of the proof, we denote by $C$ positive constants independent of $\eps$ which may vary from line to line. Combining \eqref{eq:assumptions} with Lemma \ref{complete} (ii) and arguing as in the proof of Lemma \ref{l:1} by exploiting the fact that $|(\Omega_{\eps}\cap \Omega)\setminus K_{\eps}|\to 0$ as $\eps\to 0$, we obtain
\begin{eqnarray*}\label{intersec1}
\| u_{\eps}- u\|_{H^1(\Omega_{\eps}\cap \Omega )}&
\le
& \| u_{\eps}- E_{\eps}u\|_{H^1(\Omega_{\eps}\cap \Omega )}+\| E_{\eps}u- u\|_{H^1(\Omega_{\eps}\cap \Omega )}\nonumber \\
&\le & \|u_{\eps}- E_{\eps}u\|_{H^1(\Omega_{\eps} )}+\| E_{\eps}u- u\|_{H^1((\Omega_{\eps}\cap \Omega)\setminus K_{\eps} )}\nonumber\\
&\le & C \| u_{\eps}- E_{\eps}u\|_{\eps}+ o(1)=o(1) \qquad \text{as } \eps\to 0.
\end{eqnarray*}
This completes the proof of the proposition.
\end{proof}

\section{Proof of Theorem \ref{t:main-3}} \label{s:p-t:main-3}

The proof of Theorem \ref{t:main-3} can be obtained with a
slightly different approach if compared with the proof of Theorem
\ref{t:main-1}. For simplicity we only mention the main
differences.

One can show that Lemmas \ref{l:weak-convergence}-\ref{l:compact}, with assumption \eqref{eq:conv-per} replaced by \eqref{eq:weak-L1}, still hold true with the appropriate changes: the
difference consists on the fact that the operator $S$ of Section
\ref{s:p-t:main-1} has to be replaced here by the operator
$S_\gamma$ and the scalar product $(\cdot,\cdot)_0$ in $H^1(\Omega)$ has to be replaced by the scalar product $(\cdot,\cdot)_\gamma$ defined in \eqref{eq:prod-gamma} (the same has to be done with the associated norms). For simplicity we do not write down in details the adaptations of the proofs of those three lemmas but we prefer to draw the attention only on the more delicate parts appearing in them.

We divide our explanation in three steps each of them corresponding to one of three lemmas mentioned above.

\medskip

{\bf Step 1.} In this step we state and prove the counterpart of Lemma \ref{l:weak-convergence}: let $w_\eps \EC w$, $u_\eps=S_\eps w_\eps$, $u=S_\gamma w$ and let $B_\eps$ be as in \eqref{inversodiE}. We have to show that $B_\eps u_\eps \rightharpoonup u$ in $H^1(\Omega)$ as $\eps\to 0$.

According with the notation used in the proof of Lemma \ref{l:weak-convergence}, let $\widetilde u$ be the weak limit in $H^1(\Omega)$ of $B_\eps u_\eps$ along a sequence converging to zero.

Let us proceed with the adaptation of \eqref{eq:bou-est-0}-\eqref{eq:bou-est}.

In the present setting we know that by \eqref{eq:weak-L1}-\eqref{eq:gamma}, $W_{\eps,j}\rightharpoonup \gamma \, \chi_{\partial\Omega\cap V_j}$ weakly in $L^1(\partial\Omega)$ but thanks to \eqref{eq:assumptions} and \cite[Remark 2.1]{ArBru} we actually have
\begin{equation} \label{eq:W-eps-j-conv}
W_{\eps,j}\rightharpoonup \gamma \, \chi_{\partial\Omega\cap V_j}  \qquad \text{weakly in } L^p(\partial\Omega)
\quad \text{as $\eps\to 0$}
\quad \text{for any } 1\le p<\infty \, .
\end{equation}

Looking at \eqref{eq:bou-est-0} we have that, along a sequence,
\begin{equation}\label{eq:u-cap}
  \widehat u_{\eps,j} \quad \text{is strongly convergent in $L^q(\partial\Omega)$ for any $1\le q<\tfrac{2(N-1)}{N-2}$}
\end{equation}
since $\|\widehat u_{\eps,j}\|_{H^1(\Omega)}=O(1)$ as $\eps\to 0$, where $\tfrac{2(N-1)}{N-2}$ is the critical trace exponent or equivalently the critical Sobolev exponent for the embedding $H^{1/2}(\partial\Omega)\subset L^q(\partial\Omega)$ with the usual understanding that $\frac{2(N-1)}{N-2}=\infty$ for $N=2$.
Finally we also have that $\varphi_i(\Psi_{\eps,i}(\Psi_{\eps,j}^{-1}(y)))\to \varphi_i \, \chi_{\partial\Omega\cap V_i\cap V_j}$ strongly in $L^2(\partial\Omega)$ thanks to Proposition \ref{p:7}.

Hence we may conclude that \eqref{eq:bou-est} has to be replaced by
\begin{align*}
  & \int_{\partial\Omega_\eps} u_\eps E_\eps \varphi \, dS
=\sum_{i,j=1}^{s'} \int_{\partial\Omega\cap V_i \cap V_j} \widehat
u_{\eps,j}(y) \, \varphi_i(y) W_{\eps,j}(y) \, dS+o(1)
=\int_{\partial\Omega} \widetilde u \varphi \, \gamma \, dS \, .
\end{align*}
Indeed, we recall that $\sum_{j=1}^{s'} \widehat u_{\eps,j}=B_\eps u_\eps$ on $\partial\Omega$ and that, in view of \eqref{eq:u-cap}, the trace of $B_\eps u_\eps$ is convergent to the trace of $\widetilde u$ strongly in $L^q(\partial\Omega)$ for any $1\le q<\frac{2(N-1)}{N-2}$, being $B_\eps u_\eps \rightharpoonup \widetilde u$ weakly in $H^1(\Omega)$.

This means that \eqref{eq:passaggio-1} has to be replaced by
\begin{equation} \label{eq:mod-1}
  (u_\eps,E_\eps \varphi)_\eps\to (\widetilde u,\varphi)_\gamma
\end{equation}
as $\eps\to 0$ along an appropriate sequence.

Similarly one can prove that \eqref{eq:passaggio-2} has to be replaced by
\begin{equation} \label{eq:mod-2}
  \int_{\partial\Omega_\eps} w_\eps \, E_\eps \varphi \, dS\to \int_{\partial\Omega} \gamma w \varphi\, dS
\end{equation}
as $\eps\to 0$ along an appropriate sequence.

Combining \eqref{eq:mod-1} and \eqref{eq:mod-2} we conclude that $\widetilde u$ satisfies
\begin{equation*}
  (\widetilde u,\varphi)_\gamma=\int_{\partial\Omega} \gamma w \varphi \, dS
\end{equation*}
and hence it coincides with the function $u$. The independence of $\widetilde u=u$ on the sequence converging to $0$ shows that $B_\eps u_\eps\rightharpoonup u$ in $H^1(\Omega)$ as $\eps\to 0$ in the usual sense thus completing the proof of Step 1.

\medskip

{\bf Step 2.} In this step we state and prove the counterpart of Lemma \ref{l:E-conv}. We have to show that $S_\eps \EEC S_\gamma$ as $\eps\to 0$. The main point is to prove the validity of the claim corresponding to \eqref{eq:claim-2}:
\begin{equation} \label{eq:new-claim-2}
  \int_{\partial\Omega_\eps} E_\eps w \, u_\eps \, dS\to \int_{\partial\Omega} \gamma w u\, dS=(u,u)_\gamma=\|u\|_\gamma^2
\end{equation}
where $w_\eps$, $w$, $u_\eps$ and $u$ are as in Step 1.

To this purpose one has to pass to the limit in \eqref{eq:computation}. This can be done thanks to \eqref{eq:W-eps-j-conv}, \eqref{eq:u-cap} and the fact that $w_i(\Psi_{\eps,i}(\Psi_{\eps,j}^{-1}(y)))\to w_i \, \chi_{\partial\Omega\cap V_\cap V_j}$ strongly in $L^2(\partial\Omega)$ thanks to Proposition \ref{p:7}. Following the proof of Lemma \ref{l:E-conv}, we easily obtain \eqref{eq:new-claim-2} and consequently also the convergence $u_\eps \EC u$ as $\eps\to 0$. This proves that $S_\eps \EEC S_\gamma$ as $\eps\to 0$.

\medskip

{\bf Step 3.} In this step we state and prove the counterpart of Lemma \ref{l:compact}. Let $w_\eps$ be such that $\|w_\eps\|_\eps=1$ and let $u_\eps=S_\eps w_\eps$. We have to prove the $E$-convergence of $u_\eps$ along a sequence.

We have to pass to the limit in \eqref{eq:computation-0-bis} and \eqref{eq:conv-forte} under the present assumptions.

Concerning \eqref{eq:computation-0-bis} we have that $w_{\eps,i}(\Psi_{\eps,j}^{-1}(y))\to F_i \, \chi_{\partial\Omega\cap V_i\cap V_j}$ strongly in $L^2(\partial\Omega)$, as explained in the proof of Lemma \ref{l:compact}, so that by \eqref{eq:W-eps-j-conv} we obtain
\begin{equation*}
  \int_{\partial\Omega_\eps} w_\eps \, E_\eps \varphi\, dS\to \int_{\partial\Omega} \gamma\, F\varphi \, dS
\end{equation*}
as $\eps\to 0$ along a sequence. The left hand side of \eqref{eq:u_eps} can be treated as in Step 1 thus giving rise to the identity
\begin{equation} \label{eq:id-ns}
  (\widetilde u,\varphi)_\gamma = \int_{\partial\Omega} \gamma F \varphi \, dS
\end{equation}
where $\widetilde u$ is as in Step 1.

The second crucial point in the adaptation of the proof of Lemma \ref{l:compact} is to pass to the limit in the right hand side of \eqref{eq:u-eps-bis} or equivalently in \eqref{eq:conv-forte}. We have again the two factors $w_{\eps,i}(\Psi_{\eps,j}^{-1}(y))$ and $W_{\eps,j}$ that can be treated as above but this time we have the term $\widehat u_{\eps,j}$ in place of $\varphi_j$ which however is strongly convergent in $L^q(\partial\Omega)$ for any $1\le q<\frac{2(N-1)}{N-2}$ as explained in Step 1. We then conclude that
\begin{equation*}
  \|u_\eps\|_\eps^2=\int_{\partial\Omega_\eps} w_\eps u_\eps \, dS\to \int_{\partial\Omega} \gamma F\widetilde u \, dS=\|\widetilde u\|_\gamma^2
\end{equation*}
as $\eps\to 0$ along a sequence where in the last identity we used \eqref{eq:id-ns} with $\varphi=\widetilde u$.

Proceeding as in the proof of Lemma \ref{l:compact}, we infer that $u_\eps\EC \widetilde u$ as $\eps\to 0$ along a sequence thus proving that $\{S_\eps w_\eps\}_{0<\eps\le \eps_0}$ is precompact in the sense of Definition \ref{d:precompact-ArCaLo}.

\medskip

Combining the three steps above, we conclude that $S_\eps \CC S_\gamma$ as $\eps \to 0$ thus completing the proof of the theorem as a consequence of the abstract result Theorem \ref{vaithm}.

\section{Proof of Theorem \ref{t:main-2}} \label{s:p-t:main-2}

We first state and prove two preliminary results.

\begin{lemma} \label{l:vanishing} Let $\Omega$ and $\{\Omega_\eps\}_{0<\eps\le \eps_0}$ be as in the statement of Theorem \ref{t:main-2}. Assume that \eqref{eq:th-main-2-0}
and \eqref{eq:th-main-2} hold true. Let $\{v_{\eps}\}_{0<\eps\le \eps_0}\subset H^1_{{\rm loc}}(\R^N)$ be a family of functions such that
\begin{equation*}
\int_D |\nabla v_\eps|^2 dx+\int_D v_\eps^2 \, dx =O(1), \qquad \text{as } \eps\to 0,
\end{equation*}
for some  bounded domain $D$ satisfying $D\supset \Omega\cup \bigcup_{\eps\in (0,\eps_0]} \Omega_\eps$.

Then we have:

\begin{itemize}

\item[(i)] there exist $\eps_k\downarrow 0$ and $v\in H^1(\Omega)$ such that $(v_{\eps_k})_{|\Omega}\rightharpoonup v$ weakly in $H^1(\Omega)$;

\medskip

\item[(ii)] if $\int_{\partial\Omega_{\eps_k}} v_{\eps_k}^2 \, dS=O(1)$ as $k\to +\infty$ then $v\in H^1_0(\Omega)$.

\end{itemize}

\end{lemma}

\begin{proof} For simplicity, throughout the proof of the lemma, the restrictions $(v_\eps)_{|\Omega_\eps}$ and $(v_\eps)_{|\Omega}$ will be simply denoted by $v_\eps$.

We first observe that
\begin{equation} \label{eq:bound-omega}
\|v_{\eps}\|_{H^1(\Omega)}=O(1) \qquad \text{as } \eps\to 0
\end{equation}
being $\Omega\subseteq D$.
The proof of (i) then follows by the reflexivity of $H^1(\Omega)$.

In order to prove (ii) we have to show that the function $v$ introduced in (i) belongs to $H^1_0(\Omega)$. By compactness of the trace map from $H^1(\Omega)$ into $L^2(\partial\Omega)$ we deduce that
\begin{equation} \label{eq:conv-boundary}
v_{\eps_k} \to v \qquad \text{strongly in } L^2(\partial\Omega) \qquad \text{as } k\to +\infty \, .
\end{equation}
Letting for any $j=1,\dots,s'$
\begin{align*}
\Gamma_{\eps,j}:=r_j^{-1}\left(\{(x',x_N):x'\in W_j \ \ \text{and} \ \ x_N=g_{\eps,j}(x')\}\right)\subseteq \partial \Omega_\eps \, ,
\end{align*}

\begin{align*}
\Gamma_{j}:=r_j^{-1}\left(\{(x',x_N):x'\in W_j \ \ \text{and} \ \ x_N=g_{j}(x')\}\right)\subseteq \partial \Omega \, ,
\end{align*}

\begin{align*}
w_{\eps}(x):=v_\eps(r_j^{-1}(x)) \, , \quad w(x)=v(r_j^{-1}(x)) \qquad \text{for any } x\in r_j(\Omega)
\end{align*}

\begin{align*}
& W_{t,\eps,j}:=\left\{x'\in W_j:\sqrt{1+|\nabla_{x'} \, g_{\eps,j}(x')|^2}\le t \right\}  \, , \qquad W_{t,\eps,j}^c:=W_j\setminus W_{t,\eps,j} \, ,
\end{align*}
omitting for simplicity the subindex $k$ we have
\begin{align} \label{eq:primo-passo}
O(1)& =\int_{\partial\Omega_{\eps}} v_{\eps}^2 \, dS\ge \int_{\Gamma_{\eps,j}} v_{\eps}^2 \, dS
=\int_{W_j} w_{\eps}^2(x',g_{\eps,j}(x')) \sqrt{1+|\nabla_{x'}g_{\eps,j}(x')|^2} \, dx' \\[8pt]
\notag & =\int_{W_{t,\eps,j}} w_{\eps}^2(x',g_{\eps,j}(x')) \sqrt{1+|\nabla_{x'}g_{\eps,j}(x')|^2} \, dx'
+\int_{W_{t,\eps,j}^c} w_{\eps}^2(x',g_{\eps,j}(x')) \sqrt{1+|\nabla_{x'}g_{\eps,j}(x')|^2} \, dx'  \, .
\end{align}
Inspired by \cite{DancerDaners} we estimate the first term in the second line of \eqref{eq:primo-passo}:
\begin{align} \label{eq:gs}
& \left(\int_{W_{t,\eps,j}} w_{\eps}^2(x',g_{\eps,j}(x')) \sqrt{1+|\nabla_{x'}g_{\eps,j}(x')|^2} \, dx'\right)^{1/2} \\[8pt]
&  \notag \le t^{\frac 12} \left[\left(\int_{W_j} |w_\eps(x',g_{\eps,j}(x'))-w_\eps(x',g_j(x'))|^2 dx'\right)^{\frac 12}
+\left(\int_{W_j} |w_\eps(x',g_{j}(x'))-w(x',g_j(x'))|^2 dx'\right)^{\frac 12}   \right. \\[8pt]
& \notag \qquad  \left. +\left(\int_{W_{t,\eps,j}}
w^2(x',g_j(x'))\,  dx'  \right)^{\frac 12}    \right] \, .
\end{align}
The second term in the right hand side of \eqref{eq:gs} converges to zero thanks to \eqref{eq:conv-boundary}
and the third term does the same since the measure of $W_{t,\eps,j}$ converges to zero thanks to \eqref{eq:th-main-2}. For what it concerns the first one we have
\begin{align} \label{eq:intermedio}
& \int_{W_j} |w_\eps(x',g_{\eps,j}(x'))-w_\eps(x',g_j(x'))|^2
dx'\le \int_{W_j} |g_{\eps,j}(x')-g_j(x')|
\left|\int_{g_j(x')}^{g_{\eps,j}(x')} \left|\tfrac{\partial w_\eps}{\partial t}(x',t) \right|^2 dt \right|dx' \\[8pt]
\notag & \qquad \le \|g_{\eps,j}-g_j\|_{L^\infty(W_j)} \ \int_{D} |\nabla v_\eps|^2 dx=O(1) \cdot  \|g_{\eps,j}-g_j\|_{L^\infty(W_j)}\to 0
\end{align}
as $\eps\to 0$ thanks to \eqref{eq:th-main-2-0}. This proves that the left hand side of \eqref{eq:gs} tends to zero as $\eps\to 0$ along the prescribed sequence.

Combining this fact with \eqref{eq:primo-passo}, we infer that there exists a positive constant $C$ independent of $\eps$ and $t$ such that
\begin{align} \label{eq:Cge}
C\ge \int_{W_{t,\eps,j}^c} w_{\eps}^2(x',g_{\eps,j}(x')) \sqrt{1+|\nabla_{x'}g_{\eps,j}(x')|^2} \, dx'+o(1)\ge
t \int_{W_{t,\eps,j}^c} w_{\eps}^2(x',g_{\eps,j}(x'))  \, dx'+o(1) \, .
\end{align}
We claim that
\begin{align} \label{eq:claim-W}
\int_{W_{t,\eps,j}^c} w_{\eps}^2(x',g_{\eps,j}(x'))  \, dx'=\int_{W_{t,\eps,j}^c} w^2(x',g_{j}(x'))  \, dx'+o(1)  \, .
\end{align}
Indeed by \eqref{eq:conv-boundary} and \eqref{eq:intermedio} we have
\begin{align*}
& \left| \left(\int_{W_{t,\eps,j}^c} w_{\eps}^2(x',g_{\eps,j}(x'))  \, dx'\right)^{\frac 12}
-\left(\int_{W_{t,\eps,j}^c} w^2(x',g_{j}(x'))  \, dx'\right)^{\frac 12}  \right| \\[8pt]
& \le \left(\int_{W_{t,\eps,j}^c} |w_{\eps}(x',g_{\eps,j}(x'))-w(x',g_j(x'))|^2  \, dx'\right)^{\frac 12} \\[8pt]
&  \le  \left(\int_{W_{j}} |w_{\eps}(x',g_{\eps,j}(x'))-w_\eps(x',g_{j}(x'))|^2  \, dx'\right)^{\!\!\! \frac 12}
\!\!+\! \left(\int_{W_{j}} |w_{\eps}(x',g_{j}(x'))-w(x',g_{j}(x'))|^2  \, dx'\right)^{\!\!\! \frac 12}\!\!=o(1) \, ,
\end{align*}
thus proving the validity of \eqref{eq:claim-W}.

Inserting \eqref{eq:claim-W} into \eqref{eq:Cge} and taking into account that
\begin{equation*}
\int_{W_{t,\eps,j}} w^2(x',g_{j}(x'))  \, dx'=o(1) \qquad \text{as } \eps\to 0\, ,
\end{equation*}
as explained after \eqref{eq:gs}, we obtain
\begin{align*}
C\ge t  \biggl( \int_{W_{t,\eps,j}^c} w^2(x',g_{j}(x'))  \, dx'+o(1) \biggr)+o(1)
=t \biggl(\int_{W_j}  w^2(x',g_{j}(x'))  \, dx'+o(1)\biggr)+o(1)
\end{align*}
as $\varepsilon \to 0$,
which implies
\begin{equation*}
\int_{W_j}  w^2(x',g_{j}(x'))  \, dx'\le \frac{C+o(1)}{t}+o(1)  \qquad \text{for any } t>0 \, ,
\end{equation*}
as $\varepsilon \to 0$. Here the quantities denoted by $o(1)$ do not depend on $t$.
Letting $t\to +\infty$, for any $j=1,\dots,s'$, we infer that the trace of $w$ on $r_j(\Gamma_{j})$ vanishes identically which, in turn, implies that
the trace of $v$ vanishes identically on $\Gamma_{j}$. Since the vanishing of $v$ on $\Gamma_j$ occurs for any $j\in \{1,\dots,s'\}$, this shows that $v\in H^1_0(\Omega)$.
\end{proof}

\begin{lemma} \label{l:l-i} For any integer $n\ge 0$ there exist $n+1$ functions $\varphi_0,\dots \varphi_n$ in $C^{\infty}(\R^N)$ such that, for any bounded domain $\Omega\subset \R^N$, $\varphi_0,\dots , \varphi_n$ are linearly independent when restricted to $\partial\Omega$ in the sense that if $a_0,\dots,a_n\in \R$ are such that
\begin{equation} \label{eq:con-l-i}
a_0\varphi_0(x)+\dots +a_n\varphi_n(x)=0 \qquad \text{for any } x\in \partial\Omega \, ,
\end{equation}
then $a_0=\dots=a_n=0$.
\end{lemma}

\begin{proof} We may construct explicitly the functions $\varphi_0,\dots,\varphi_n$ choosing for example
\begin{align} \label{eq:def-phi}
& \varphi_0(x)=1 \quad \text{for any } x\in \R^N  \, , \\[7pt]
\notag & \varphi_k(x)=x_1^k  \quad \text{and for any } x=(x_1,\dots,x_N)\in \R^N, \,  \text{for any } k\in \{1,\dots,n\} \, .
\end{align}
Let us prove that this functions are linearly independent when restricted to $\partial\Omega$.

To this purpose, let $a_0,\dots,a_n\in \R$ be such that
\eqref{eq:con-l-i} holds true and define the polynomial in one variable $P(t)=a_0+\sum_{k=1}^n a_k t^k$. Consider also the corresponding function $\Phi:\R^N\to \R$ defined by $\Phi(x)=P(x_1)$ for any $x=(x_1,\dots,x_N)\in \R^N$.

Suppose by contradiction that $(a_0,\dots,a_n)\neq (0,\dots,0)$ so that $P$ is not the null polynomial.

Since the degree of the polynomial $P$ is at most $n$, let $t_1,\dots,t_m$, $0\le m\le n$, be its real roots with the convention that when $m=0$ the polynomial $P$ has no real roots. Then we have that
\begin{equation} \label{eq:null-set}
\{x\in \R^N:\Phi(x)=0\}=\bigcup_{k=1}^m \Pi_k
\end{equation}
with $\Pi_k=\left\{x=(x_1,\dots,x_N)\in \R^N: x_1=t_k\right\}$.

By \eqref{eq:con-l-i}, \eqref{eq:def-phi} and \eqref{eq:null-set}, we deduce that $\partial\Omega\subset \bigcup_{k=1}^m \Pi_k$. This means that the boundary of the domain $\Omega$ is contained in a finite union of parallel hyperplanes of $\R^N$ and this contradicts the fact that $\Omega$ is bounded.
Indeed, there exists $c\in \R\setminus \{t_1,\dots,t_m\}$ such that the hyperplane $\Pi_c:=\{x=(x_1,\dots,x_N):x_1=c\}$ intersects the set $\overline\Omega$, hence also $\partial \Omega$ if $\Omega$ is bounded, but this contradicts the condition  $\partial\Omega\subset \bigcup_{k=1}^m \Pi_k$ found above.
\end{proof}

{\bf End of the proof of Theorem \ref{t:main-2}.} Let $n\ge 1$ and consider the corresponding eigenvalues $\lambda_n^\eps$. Let $\varphi_0,\dots,\varphi_n$ be $n+1$ functions as in Lemma \ref{l:l-i}. We introduce the following family of subspaces of $H^1(\Omega_\eps)$
\begin{equation*}
V_\eps:={\rm span}\{(\varphi_0)_{|\Omega_\eps},\dots,(\varphi_n)_{|\Omega_\eps}\} \, .
\end{equation*}
By Lemma \ref{l:l-i} we know that $\varphi_0,\dots,\varphi_n$ are linearly independent as functions defined over $\partial\Omega_\eps$ and in particular as functions defined over $\Omega_\eps$. This implies ${\rm dim}(V_\eps)=n+1$ and $V_\eps\cap H^1_0(\Omega_\eps)=\{0\}$ for any $\eps\in (0,\eps_0]$.

Therefore by (ii)-(iii) in Proposition \ref{p:min-max}, we infer
\begin{equation} \label{eq:est-above}
\lambda_n^\eps\le \max_{v\in V_\eps\setminus\{0\}} \frac{\int_{\Omega_\eps} |\nabla v|^2 dx}{\int_{\partial\Omega_\eps} v^2 dS} \, .
\end{equation}
Let $a_0^\eps,\dots,a_n^\eps\in \R$ be such that $v_\eps:=\sum_{j=0}^n a_j^\eps \varphi_j$ satisfies
\begin{equation} \label{eq:achieved}
\frac{\int_{\Omega_\eps} |\nabla v_\eps|^2 dx}{\int_{\partial\Omega_\eps} v_\eps^2 dS}
=\max_{v\in V_\eps\setminus\{0\}} \frac{\int_{\Omega_\eps} |\nabla v|^2 dx}{\int_{\partial\Omega_\eps} v^2 dS} \, .
\end{equation}
It is not restrictive, up to normalization, to assume that $\sum_{j=0}^n (a_j^\eps)^2=1$.

Let $D$ be a bounded domain such that $D\supset \Omega\cup \bigcup_{\eps\in (0,\eps_0]} \Omega_\eps$ whose existence easily follows by \eqref{eq:th-main-2-0}.

First of all, by the Cauchy-Schwarz inequality we have
\begin{align} \label{eq:bound-grad}
& \int_{D} |\nabla v_\eps|^2 dx\le \sum_{j=0}^{n} \int_{D} |\nabla \varphi_j|^2 dx =O(1) \qquad
\text{as } \eps\to 0
\, , \\[8pt]
\notag & \int_{D} v_\eps^2 dx\le \sum_{j=0}^{n}  \int_{D} \varphi_j^2 \, dx =O(1) \qquad \text{as } \eps\to 0 \, .
\end{align}
We claim that
\begin{equation} \label{eq:claim-3}
\int_{\partial\Omega_\eps} v_\eps^2 dS\to +\infty \qquad \text{as } \eps\to 0 \, .
\end{equation}
Suppose by contradiction that there exists a sequence
$\eps_k\downarrow 0$ such that $\int_{\partial\Omega_{\eps_k}} v_{\eps_k}^2 dS$ remains bounded as $k\to +\infty$.

Then, by \eqref{eq:bound-grad} and Lemma \ref{l:vanishing}, we
deduce that there exists $v\in H^1(\Omega)$ such that
\begin{equation} \label{eq:weak-conv}
v_{\eps_k} \rightharpoonup v \qquad \text{weakly in } H^1(\Omega)
\end{equation}
and moreover $v\in H^1_0(\Omega)$. Actually the weak convergence in \eqref{eq:weak-conv} is strong since the sequence of functions $\{(v_{\eps_k})_{|\Omega}\}$ lives in a finite dimensional space. By possibly passing to a subsequence,
we may assume that there exist $a_0,\dots,a_n\in \R$ such that
\begin{equation} \label{eq:conv-coeff}
a_j^{\eps_k}\to a_j \qquad \text{as } k\to +\infty \, , \ \text{for any } j\in \{0,\dots,n\} \, .
\end{equation}
Clearly the normalization condition $\sum_{j=0}^n (a_j^{\eps_k})^2=1$ implies
\begin{equation} \label{eq:normalization}
 \sum_{j=0}^n (a_j)^2=1 \, .
\end{equation}

Combining \eqref{eq:weak-conv} and \eqref{eq:conv-coeff} we deduce that $v=\sum_{j=0}^n a_j\varphi_j$.
Since $v\in H^1_0(\Omega)$, we have that condition \eqref{eq:con-l-i} on $\partial\Omega$ is satisfied and hence we infer
$$
a_0=\dots=a_n=0
$$
thus contradicting \eqref{eq:normalization}. This completes the proof of claim \eqref{eq:claim-3}.

The proof of the theorem now follows immediately combining \eqref{eq:est-above}, \eqref{eq:achieved}, \eqref{eq:bound-grad} and \eqref{eq:claim-3}.

\section{Proof of Proposition \ref{p:comparison}} \label{s:p-p:comparison}

{\bf Proof of (i).} It can be proved with a simple computation.

\medskip

{\bf Proof of (ii).} The fact that $g_\eps$ converges uniformly to
zero as $\eps\to 0$ and that the first order partial derivatives
remain uniformly bounded as $\eps\to 0$, follows immediately
after a simple computation. It remains to prove the weak
$L^1$-convergence of the surface element of $\partial\Omega_\eps$.
i.e. $\sqrt{1+|\nabla_{x'}b(x'/\eps)|^2}$. This is a standard fact for which we  refer {e.g.,} to  \cite{lukkawall}.  However, since we need to introduce some notation for the proof of part (iii) and, having that notation,  the proof takes only a few lines, we include a proof here for the convenience of the reader.

We actually prove the weak convergence in $L^p(W)$ for any
$1<p<\infty$, i.e.
\begin{equation} \label{eq:weak-Lp}
\int_W \sqrt{1+|\nabla_{x'}b(x'/\eps)|^2} \, \varphi(x')\, dx'\to
C_b \int_W \varphi(x')\, dx' \qquad \text{for any } \varphi\in
L^{\frac p{p-1}}(W)
\end{equation}
where $C_b:=\int_Y \sqrt{1+|\nabla_{x'}b(y')|^2} \, dy'$.

By density of $C^\infty_c(W)$ in $L^p(W)$ it is enough to prove
\eqref{eq:weak-Lp} for any $\varphi\in C^\infty_c(W)$.

For any $\eps>0$ let us define the family $\mathcal Q_\eps$ of all
$(N-1)$-dimensional closed cubes where every $Q\in \mathcal
Q_\eps$ is in the form $Q=\overline{\eps(Y+z)}$ with $z\in \mathbb
Z^{N-1}$. For every $\eps>0$ let $n_\eps$ be the number of cubes in
$\mathcal Q_\eps$ contained in $W$; let us denote these cubes by
$Q_{\eps,1}\, , \, \dots \, , \, Q_{\eps,n_\eps}$ so that
$Q_{\eps,j}=\overline{\eps(Y+z_{\eps,j})}$ with $z_{\eps,j}\in
\mathbb Z^{N-1}$ for any $j\in \{1,\dots,n_\eps\}$.

Put $F_\eps=\cup_{j=1}^{n_\eps} Q_{\eps,j}$; from this definition
it easily follows that for every $\eps_1>0$ there exists $0<\eps_2<\eps_1$ such that $F_{\eps}\supseteq F_{\eps_1}$ for any $0<\eps<\eps_2$ and moreover $\cup_{\eps>0} F_\eps=W$.
This implies
\begin{equation} \label{eq:HN-1}
\lim_{\eps\to 0} \mathcal H^{N-1}(F_\eps)=\mathcal H^{N-1}(W) \, .
\end{equation}
On the other hand, being $Q_{\eps,j}$ cubes with disjoint
interiors, we deduce that
$$
\mathcal H^{N-1}(F_\eps)=\sum_{j=1}^{n_\eps} \mathcal
H^{N-1}(Q_{\eps,j})=n_\eps \, \eps^{N-1}
$$
which combined with \eqref{eq:HN-1} gives
\begin{equation} \label{eq:est-neps}
\lim_{\eps \to 0} n_\eps \, \eps^{N-1}=\mathcal H^{N-1}(W) \, .
\end{equation}

Being $\varphi\in C^\infty_c(W)$, it is not restrictive to assume
that ${\rm supp}(\varphi)\subset F_\eps$ for all $\eps$ small
enough.

Using a change of variables, exploiting the $Y$-periodicity of $b$,
the Lipschitzianity of $\varphi$, as a consequence of its
smoothness, and \eqref{eq:est-neps}, we obtain
\begin{align*}
& \int_W \sqrt{1+|\nabla_{x'}b(x'/\eps)|^2} \, \varphi(x')\, dx'
=\sum_{j=1}^{n_\eps}
\int_{Q_{\eps,j}}\sqrt{1+|\nabla_{x'}b(x'/\eps)|^2} \,
\varphi(x')\, dx' \\[8pt]
& \qquad =\sum_{j=1}^{n_\eps}
\int_{Y}\sqrt{1+|\nabla_{x'}b(y')|^2} \, \varphi(\eps y'+\eps
z_{\eps,j}) \, \eps^{N-1}\, dy' \\[8pt]
& \qquad =\sum_{j=1}^{n_\eps}
\int_{Y}\sqrt{1+|\nabla_{x'}b(y')|^2} \, \varphi(\eps z_{\eps,j})
\, \eps^{N-1}\, dy'+o(1) \\[8pt]
& \qquad =\left[\int_{Y}\sqrt{1+|\nabla_{x'}b(y')|^2} \, dy' \cdot
\sum_{j=1}^{n_\eps} \varphi(\eps z_{\eps,j})\,  \mathcal
H^{N-1}(Q_{\eps,j})\right]+o(1) \, .
\end{align*}
Taking into account that $\sum_{j=1}^{n_\eps} \varphi(\eps
z_{\eps,j})\,  \mathcal H^{N-1}(Q_{\eps,j})$ represents a Cauchy-Riemann
sum for $\int_W \varphi(x')\, dx'$, letting $\eps\to 0$, the proof
of \eqref{eq:weak-Lp} easily follows.

Clearly, the validity of the weak convergence in $L^p(W)$ for any
$1<p<\infty$ implies the validity of the weak convergence in
$L^1(W)$ thus completing the proof of (ii).

\bigskip

{\bf Proof of (iii).} The validity \eqref{eq:th-main-2-0} follows
immediately from the definition of $g_\eps$. It remains to prove
the validity of \eqref{eq:th-main-2}.

Let $\mathcal Q_\eps$, $n_\eps$, $Q_{\eps,1}\, , \, \dots \, , \,
Q_{\eps,n_\eps}$ and $F_\eps$ be as in the proof of (ii).

By \eqref{eq:HN-1}, \eqref{eq:est-neps} and the $Y$-periodicity of $b$ we infer that
\begin{align} \label{eq:g-eps}
& \mathcal H^{N-1}\left(\{x'\in W:|\nabla_{x'} g_\eps(x')|\le t\}\right)
=\mathcal H^{N-1}\left(\{x'\in W:|\nabla_{x'} \, b(x'/\eps)|\le \eps^{1-\alpha} t\}\right) \\[7pt]
\notag & \le \mathcal H^{N-1}(W\setminus F_\eps)+\mathcal H^{N-1}\left(\{x'\in F_\eps:|\nabla_{x'} \, b(x'/\eps)|\le \eps^{1-\alpha} t\}\right) \\[7pt]
\notag & =o(1)+\sum_{j=1}^{n_\eps} \mathcal H^{N-1}\left(\{x'\in Q_{\eps,j}:|\nabla_{x'} \, b(x'/\eps)|\le \eps^{1-\alpha} t\}\right) \\[7pt]
\notag & =o(1)+\sum_{j=1}^{n_\eps} \eps^{N-1} \, \mathcal H^{N-1}\left(\{y'\in \eps^{-1} Q_{\eps,j}:|\nabla_{x'} \, b(y')|\le \eps^{1-\alpha} t\}\right) \\[7pt]
\notag & =o(1)+\sum_{j=1}^{n_\eps} \eps^{N-1} \, \mathcal H^{N-1}\left(\{y'\in Y:|\nabla_{x'} \, b(y')|\le \eps^{1-\alpha} t\}\right) \\[7pt]
\notag & =o(1)+\eps^{N-1} n_\eps \, \mathcal H^{N-1}\left(\{y'\in Y:|\nabla_{x'} \, b(y')|\le \eps^{1-\alpha} t\}\right) \\[7pt]
\notag & =o(1)+\left(\mathcal H^{N-1}(W)+o(1)\right) \, \mathcal H^{N-1}\left(\{y'\in Y:|\nabla_{x'} \, b(y')|\le \eps^{1-\alpha} t\}\right)\to 0
\end{align}
as $\eps\to 0$ since $0<\alpha<1$ and
$$
\lim_{\delta\to 0} \mathcal H^{N-1}\left(\{y'\in Y:|\nabla_{x'} \, b(y')|\le \delta\}\right)=0
$$
being $\mathcal H^{N-1}\left(\{y'\in Y:|\nabla_{x'} \, b(y')|=0\}\right)=0$ by \eqref{eq:Lebesgue}.

By \eqref{eq:g-eps} and the fact that for any $t>0$ we have
$$
\left\{x'\in W:\sqrt{1+|\nabla_{x'} \, g_\eps(x')|^2}\le t\right\}\subseteq \left\{x'\in W:|\nabla_{x'} \, g_\eps(x')|\le t\right\} \, ,
$$
the validity of \eqref{eq:th-main-2} follows immediately.

\section{Proof of Theorem \ref{t:tricotomia}} \label{s:t-t:tricotomia}

{\bf Proof of (i).} Since $\alpha>1$ it is sufficient to observe
that thanks to Proposition \ref{p:comparison} (i), we may proceed
exactly as in the proof of Theorem \ref{t:main-1} with the
advantage that here we have to deal with an atlas with only one
chart.

\bigskip

{\bf Proof of (ii).} Similarly to (i), since $\alpha=1$, we may
exploit Proposition \ref{p:comparison} (ii) and proceed as in the
proof of Theorem \ref{t:main-3}.

\bigskip

{\bf Proof of (iii).} The proof can be obtained essentially like
the one of Theorem \ref{t:main-2} once we exploit Proposition
\ref{p:comparison} (iii).

We first observe that under the assumptions of Theorem
\ref{t:tricotomia} the conclusions of Lemma \ref{l:vanishing}
still hold true. Then we need a revised version of Lemma
\ref{l:l-i} which takes into account the homogeneous Dirichlet
boundary conditions on $\Sigma_\eps$ appearing in
\eqref{eq:Steklov-modified}.

\begin{lemma} \label{l:l-i-2} Let $\{\Omega_\eps\}_{\eps\ge 0}$ be as in Theorem \ref{t:tricotomia} (iii) and let $\Sigma_\eps$ the set defined in \eqref{eq:Gamma-Sigma}. For any integer $n\ge 0$ there exist $n+1$ functions $\psi_0,\dots, \psi_n$ in $C^\infty(\R^N)$ (independent of $\eps$) with null traces on $\Sigma_\eps$ such that, for any $\eps\ge 0$, $\psi_0,\dots \psi_n$ are linearly independent when restricted to $\partial\Omega_\eps$ in the sense that if $a_0,\dots,a_n\in \R$ are such that
\begin{equation} \label{eq:con-l-i-2}
a_0\psi_0(x)+\dots +a_n\psi_n(x)=0 \qquad \text{for any } x\in \partial\Omega_\eps \, ,
\end{equation}
then $a_0=\dots=a_n=0$.
\end{lemma}

\begin{proof} Let $\varphi_0,\dots,\varphi_n$ be the functions defined in \eqref{eq:def-phi}. We now introduce the two following cutoff functions denoted by $\eta_1$ and $\eta_2$ respectively.

Let $W$ be as in \eqref{eq:Omega-eps}. We choose $\eta_1\in C^\infty_c(W)$, $\eta_1\ge 0$ in $W$, $\eta_1\equiv 1$ in $B'$ where $B'\subset W$ is an open ball of $\R^{N-1}$. We then choose $\eta_2\in C^\infty(\R)$ with $\eta_2(t)=0$ for any $t\le -1$ and $\eta_2(t)=1$ for any $t\ge -\frac 12$. Now, for any $k\in \{0,\dots,n\}$ we define
\begin{equation} \label{eq:psi-k}
\psi_k(x',x_N):=\varphi_k(x',x_N)\, \eta_1(x') \, \eta_2(x_N)   \qquad \text{for any } (x',x_N)\in \R^N \, .
\end{equation}
Clearly $\psi_0,\dots,\psi_n$ vanish on $\Sigma_\eps$.

It remains to prove the linear independence of the functions $\psi_0,\dots,\psi_n$ on $\partial\Omega_\eps$.
Let $a_0,\dots,a_n\in \R$ be as in \eqref{eq:con-l-i-2} and as in the proof of Lemma \ref{l:l-i} we define $P(t)=a_0+\sum_{k=1}^n a_k t^k$.

Denote by $S$ the segment $B' \cap \{(x_1,0,\dots,0):x_1\in \R\}$. Then by \eqref{eq:con-l-i-2} and \eqref{eq:psi-k} we have
\begin{align} \label{eq:segment}
0& =\eta_1(x')\, \eta_2(x_N)\sum_{k=0}^n a_k \varphi_k(x',x_N)=\sum_{k=0}^n a_k \varphi_k(x',x_N)  \\[7pt]
\notag & =a_0+\sum_{k=1}^n a_k x_1^k \qquad \text{for any } x\in \{(x',g_\eps(x')):x'\in S\}  \, .
\end{align}
Let $I$ be the open interval defined by $I=\{x_1\in \R:(x_1,0,\dots,0)\in S\}$. Hence by \eqref{eq:segment} the polynomial $P$ vanishes on the interval $I$ which means that $P$ is the null polynomial and in particular $a_0=\dots=a_n=0$. This completes the proof of the lemma.
\end{proof}

We now observe that the eigenvalues of \eqref{eq:Steklov-modified} in $\Omega_\eps$ admit the following
variational characterization
\begin{equation} \label{eq:lambda-n-eps}
\lambda_n(\eps)=\min_{V\in \mathcal V_{\eps,n}} \, \sup_{v\in
V\setminus H^1_0(\Omega_\eps)} \frac{\int_{\Omega_\eps} |\nabla
v|^2 dx}{\int_{\partial\Omega_\eps} v^2 \, dS}
\end{equation}
where
\begin{align*}
& \mathcal V_{\eps,n}:=\{V\subseteq H^1_{0,\Sigma_\eps}(\Omega_\eps):{\rm dim}(V)=n+1 \ \text{and} \ V\nsubseteq H^1_0(\Omega_\eps) \} \, , \\[7pt]
& H^1_{0,\Sigma_\eps}(\Omega_\eps):=\{v\in H^1(\Omega_\eps):v=0 \ \text{on } \Sigma_\eps\} \, ,
\end{align*}
as one can verify by proceeding as we did for Proposition \ref{p:min-max}.

Once we have the variational characterization
\eqref{eq:lambda-n-eps}, we can proceed as in the proof of
Theorem \ref{t:main-2} using here as $V_\eps$ the space
$
V_\eps:={\rm span} \{(\psi_0)_{|\Omega_\eps},\dots,
(\psi_n)_{|\Omega_\eps} \}
$,
where $\psi_0,\dots,\psi_n$ are the functions introduced in the
statement of Lemma \ref{l:l-i-2}. The proof of (iii) then follows.

\bigskip

{\bf Acknowledgments} The authors are members of the Gruppo
Nazionale per l'Analisi Matematica, la Probabilit\`{a} e le loro
Applicazioni (GNAMPA) of the Istituto Nazionale di Alta Matematica
(INdAM). The first author acknowledges partial financial support
from the PRIN project 2017 ``Direct and inverse problems for
partial differential equations: theoretical aspects and
applications''. The authors acknowledge partial financial support
from the INDAM - GNAMPA project 2019 ``Analisi spettrale per
operatori ellittici del secondo e quarto ordine con condizioni al
contorno di tipo Steklov o di tipo parzialmente incernierato''.

This research was partially supported by the research project
 BIRD191739/19  ``Sensitivity analysis of partial differential equations in the mathematical theory of electromagnetism'' of the University of Padova
and by
the research project ``Metodi e modelli per la matematica e le sue
applicazioni alle scienze, alla tecnologia e alla formazione''
Progetto di Ateneo 2019 of the University of Piemonte Orientale
``Amedeo Avogadro''.

\end{document}